\documentclass [11pt]{article}
\usepackage{amsmath,amscd,amsthm,amssymb}
\usepackage{graphicx}
\usepackage{ulem}

\newtheorem{theorem}{Theorem}[section]
\newtheorem{lemma}[theorem]{Lemma}

\newtheorem{construction}[theorem]{Construction}

\newtheorem{example}{Example}[section]

\newcommand{\G}{\ensuremath{\mathcal{G}}}
\newcommand{\M}{\ensuremath{\mathcal{M}}}
\newcommand{\E}{\ensuremath{\mathcal{E}}}
\newcommand{\D}{\ensuremath{\mathcal{D}}}
\newcommand{\C}{\ensuremath{\mathcal{C}}}
\newcommand{\B}{\ensuremath{\mathcal{B}}}
\newcommand{\A}{\ensuremath{\mathcal{A}}}
\newcommand{\F}{\ensuremath{\mathcal{F}}}
\newcommand{\zed}{\ensuremath{\mathbb{Z}}}
\newcommand{\p}{\ensuremath{{\bf Proof:\ }}}
\newcommand{\Pt}{{\cal P}}
\newcommand{\ucite}[1]{\textsuperscript{\cite{#1}}}
\begin{document}
\title{More results on large sets of Kirkman triple systems\footnote{Supported by NSFC grant 11771119.}}
%Some new results on Steiner quadruple systems ~~~~~~~~~~~~~~~~~~~~~with resolvable derived designs
\author{Yan Liu, Jianguo Lei\thanks{Corresponding author. Email: leijg1964@hebtu.edu.cn}\\\scriptsize{School of Mathematical Sciences, Hebei Normal University, Shijiazhuang 050024, P. R. China}}
\date{}
\maketitle
\def\binom#1#2{{#1\choose#2}}
\def \lb{\lbrack}
\def \rb{\rbrack}
\def \l{\langle}
\def \r{\rangle}
\def \m{\equiv}
\def \lb{\lbrack}
\def\rb{\rbrack}
\def\sub{\subseteq}
\def\inf{\infty} \def\om{\omega}
\def\ba{\bigcap}
\def\bu{\bigcup}
\def\sm{\setminus}
\def\ct{{\cal T}}
\def\ca{{\cal A}}
\def\cb{{\cal B}}
\def \o{\overline}
\def \a{\alpha}
\begin{minipage}[t]{13.5cm}
{\bf Abstract:} The existence of large sets of Kirkman triple systems (LKTSs) is  one of the best-known open problems in combinatorial
design theory. Steiner quadruple systems with resolvable derived designs (RDSQSs)  play an important role in the recursive constructions of LKTSs.
In this paper, we introduce  a special combinatorial structure  RDSQS$^{\ast}(v)$ and use it to present a construction for RDSQS$(4v)$.
 As a consequence, some new infinite families of  LKTSs are given.

\vspace{0.3cm}

\vspace{0.3cm}

\end{minipage}

\baselineskip=17pt
\section{Introduction}

Let $X$ be a $v$-element set and $K$ a set of positive integers. A \textit{$t$-wise  balanced  design} S$(t, K, v)$ is a pair $(X, \mathcal{B})$ where $\mathcal{B}$ is a collection of subsets of $X$ (called {\it blocks}) with the property that the size of every block is in the set $K$ and every $t$-element subset of $X$ is contained in exactly one block. An S$(t, \{k\}, v)$ is denoted by S$(t, k, v)$ and called a \textit{Steiner system}. An S$(2, 3, v)$ is called a \textit{Steiner triple system}, denoted by STS$(v)$. An S$(3, 4, v)$ is called a \textit{Steiner quadruple system} and denoted by SQS$(v)$. It is well known that there is an SQS$(v)$ if and only if $v \equiv 2, 4$ (mod~6) \cite{1960hanani}.

Let $(X, \mathcal{B})$ be an S$(t, k, v)$. If there exists a partition $\Gamma = \{P_{1}, P_{2}, . . . , P_r\}$ of $\mathcal{B}$ such that each  $P_{i}$ forms a \textit{parallel class}, i.e. a partition of $X$, then the S$(t, k, v)$ is called \textit{resolvable} and $\Gamma$ is called a \textit{resolution}. A resolvable STS$(v)$ is called a \textit{Kirkman triple system} and  denoted by KTS$(v)$. It is well known that a KTS$(v)$ exists if and only if $v \equiv 3 $ (mod~6) \cite{1971ray}.

 Two STS$(v)$s on the same set $X$ are said to be \textit{disjoint} if they have no triples in common. A partition of all the triples of a $v$-element set $X$ into $v-2$  disjoint STS$(v)$s, denoted by LSTS$(v)$, is called a \textit{large set} of STS$(v)$s. An LSTS($v$) is known to exist if and only if  $v \equiv 1, 3 $ (mod~6) and $v \neq 7 $ \cite{1983lu,1984lu,1991teirlinck}.

Furthermore, if each STS$(v)$ in an LSTS$(v)$ is a KTS$(v)$, then it is called a \textit{large set of Kirkman triple systems} and denoted by LKTS$(v)$.
 The existence of LKTSs remains an open problem though much work has been done
by many researchers. The existence of an LKTS(15) was shown by Denniston \cite{1974denniston}, and more small orders were given in \cite{1999changge,1979denniston1,2004ge,2017zheng1,2017zheng2,2010zhou}. Denniston \cite{1979denniston2} used transitive KTSs to give a tripling construction for LKTSs, which was generalized to a product construction by Lei \cite{2002lei1}. Zhang and Zhu \cite{2002zhangzhu,2003zhangzhu} improved Denniston's tripling construction and Lei's product construction by removing the condition of transitive KTS. Lei \cite{2002lei2,2004lei} also displayed a recursive construction for LKTSs based on 3-wise balanced designs.  Lately, Xu and Ji \cite{2021xu} used a (2, 1)-RSQS$^{\ast}(4n)$ to construct an LKTS($2^{2n+1} + 1$). More constructions and results on LKTSs were displayed in Chang and Zhou's survey paper \cite{2013chang}.

In the process of studying LKTSs, many combinatorial configurations have been proposed in order to present some recursive constructions for LKTSs. The Steiner quadruple system with resolvable derived designs (RDSQS) is one of these special configurations. The RDSQS was firstly used by Chang and Zhou in \cite{2013chang}, where its main application in constructing LKTSs is restated after Yuan and Kang \cite{2008yuankanganother}.

Let $(X, \mathcal{B})$ be an S$(t, k, v)$. For any $x\in X$, let $\mathcal{B}_{x} = \{ B\setminus \{x\}: x \in B, B \in\mathcal{B}\}$, then $(X\setminus \{x\}, \mathcal{B}_{x})$ is an S$(t-1, k-1, v-1)$, which is said to be the $derived~ design$ at the point $x$. An S$(t, k, v)$ with resolvable derived designs, abbreviated to RDS$(t, k, v)$, refers to an S$(t, k, v)$ whose derived design at every point is resolvable. An  RDSQS$(v)$ refers to an RDS$(3, 4, v)$, i.e., its derived design at every point is a KTS$(v-1)$. Necessarily, the existence of an  RDSQS$(v)$ requires that $v \equiv 4~$(mod~6).

\begin{theorem}{\rm{\cite{2013chang,2008yuankanganother}\label{1lkts}}}
If there exists an RDSQS$(v +1)$, then there exists an $LKTS(3v)$.
\end{theorem}

In this paper, three infinite families of RDSQSs are given. As a consequence, some new infinite families of LKTSs are given.

The rest of this paper is organized as follows. In Section 2, we use an auxiliary design  named  group divisible  design with resolvable derived designs to construct an RDSQS$(v)$. In Section 3, a special combinatorial structure denoted by RDSQS$^{\ast}(v)$ is introduced. We use it to construct an RDSQS($4v$). In the last section we combine the known constructions to update and expand our results on RDSQSs and LKTSs.

\section{RDSQSs and related designs}

In this section, we mainly establish some results of RDSQS$(v)$  by using the combinatorial
structure group divisible  design with resolvable derived designs (RDGDD). In what follows,
 $Z_v=\{0,1,\ldots,v-1\}$ denotes the integer modulo $v$ residual class additive group.

Let $v$ and $t$ be positive integers and $K$   a set of positive integers. A \textit{group divisible $t$-design} of order $v$ with block sizes from $K$, denoted by GDD($t, K, v$), is a triple $(X, \mathcal{G}, \mathcal{B})$ with the following
properties:

(1) $X$ is a set of $v$ elements (called \textit{points});

(2) $\mathcal{G}=\{G_1, G_2,\dots\}$ is a set of non-empty subsets (called \textit{groups}) of $X$ which partition $X$;

(3) $\mathcal{B}$ is a family of  subsets of $X$  (called \textit{blocks}), each of cardinality from $K$, such that each block intersects any given group
in at most one point, and each $t$-element set of points from $t$ distinct groups is contained in exactly one block.\vspace{0.2cm}

The \textit{type} of a GDD is defined to be the multiset $T = \{|G|: G \in \mathcal{G}\}$ of group sizes. We also use $a^{i}b^{j}c^{k}\cdots$ to denote the type, which means that in the multiset there are $i$ occurrences of $a$, $j$ occurrences of $b$, etc. We often write GDD$(t, k, v)$ instead of GDD$(t, \{k\}, v)$. A  GDD$(t, k, kn)$ of type $n^{k}$ is also called a \textit{transversal $t$-design} and also denoted by TD$(t, k, n)$.

\begin{lemma} {\rm{\cite{1990mill}}}\label{S_gdd}
If there is an S$(t, K, v)$, and for every $k\in K$, there exists a GDD$(t, K^{'}, gk)$ of type $g^{k}$, then there exists a GDD$(t, K^{'}, gv)$ of type $g^{v}$.
\end{lemma}

Let $(X,\mathcal{G},\mathcal{B})$ be a GDD$(t, k, gn)$ of type $g^{n}$. If $\mathcal{B}$ can be partitioned into some parallel classes, then the GDD$(t, k, gn)$ is called resolvable. For any $x \in G$ and $G \in \mathcal{G}$, let $\mathcal{B}_{x} = \{ B\setminus \{x\}: x \in B, B \in\mathcal{B}\}$, then $(X\setminus G,\mathcal{G} \setminus \{G\},\mathcal{B}_{x})$ is a GDD$(t-1, k-1, g(n-1))$ of type $g^{n-1}$, which is said to be the $derived ~design$ at the point $x$. A GDD$(t, k, v)$ with resolvable derived designs, denoted by RDGDD$(t, k, v)$, refers to a GDD$(t, k, v)$ whose derived design at every point is resolvable.
Similarly,  RDTD$(t, k, n)$ represents a TD$(t, k, n)$ whose derived design at every point is resolvable.

\begin{lemma} \label{gdd3_8}
There exists an $RDGDD(3, 4, 24)$ of type $3^{8}$.
\end{lemma}
\begin{proof}
We will construct a  GDD$(3, 4, 24)$ of type $3^{8}$ on $X = ( Z_{7}\cup\{ \infty \} ) \times Z_{3}$ with group set $\{\{ x \} \times Z_{3} : x \in Z_{7}\cup \{ \infty \} \}$.

We first construct an SQS$(8)$ on $Z_{7}\cup\{ \infty \}$ with block set $\mathcal{A}$, the base blocks under the group $Z_{7}$ are
$\{\infty,~0,~1,~3\}, \{0,~1,~2,~5\}.$ For every block $A=\{a_{0},a_{1},a_{2},a_{3}\} \in \mathcal{A}$, there exists a TD$(3,4,3)$ on $A\times Z_{3}$ with group set $\{\{x\}\times Z_{3}: x\in A\}$.
%For convenience, we set
%$$\{(a_{0},x),(a_{1},y),(a_{2},z),(a_{3},u)\}\triangleq (a_{0},a_{1},a_{2},a_{3})\times(x,y,z,u).$$
By Lemma \ref{S_gdd}, we can get a GDD$(3, 4, 24)$ of type $3^{8}$ with  block set $\mathcal{B}$ as follows:
$$\mathcal{B}=\{\{(a_{0},x),(a_{1},y),(a_{2},z),(a_{3},u)\}: (a_{0},a_{1},a_{2},a_{3})~ {\rm and}~ (x,y,z,u)~ {\rm are~ listed~ below}\}.$$

\begin{center}
\begin{tabular}{|c|c|}
\hline
$(a_{0},a_{1},a_{2},a_{3})$&$(x,y,z,u), x, y, z, u\in Z_{3}$\\ \hline\hline

$(\infty,~i,~1+i,~3+i)$& $x + y + z + u$ $\equiv m$~(mod~3)\\

$i$ = 0,~1&$(x, m) \in$ \{(0, 0), (1, 2), (2, 1)\}\\ \hline

$(\infty, ~i, ~1+i, ~3+i)$& \\

$i$ = 2,~3,~4,~5,~6& \raisebox{1.5ex}[0pt]{$x + y + z + u \equiv$ 0~ (mod~3)}\\\hline

$(i, ~1+i, ~2+i, ~5+i)$& \\

$i$ = 0, 2, 6& \raisebox{1.5ex}[0pt]{$x + y - z - u \equiv$ 0~ (mod~3)}\\\hline

$(i, ~1+i, ~2+i, ~5+i)$& \\

$i$ = 1,~3&  \raisebox{1.5ex}[0pt]{$x + y + z + u \equiv$ 0~ (mod~3)}\\\hline

& $x + y - z + u \equiv m$~(mod~3)\\

\raisebox{1.5ex}[0pt]{(4,~5,~6,~2)}&$(x, m) \in$ \{(0, 0), (1, 2), (2, 1)\}\\\hline

(5,~6,~0,~3)&$ x + y - z + u \equiv$ 0~ (mod~3)\\\hline
\end{tabular}
\end{center}
\vspace{0.2cm}

We can check that the derived design of the GDD at every point is resolvable.  The resolutions of derived designs at each point of $X$ are listed in Appendix A. Thus, $(X, \mathcal{B})$ forms an RDGDD$(3, 4, 24)$ of type $3^{8}$.
\end{proof}

\begin{lemma}\rm{}\label{gdd3_14}
There exists an $RDGDD(3, 4, 42)$ of type $3^{14}$.
\end{lemma}
\begin{proof}
We will construct a  GDD$(3, 4, 42)$ of type $3^{14}$ on $X = Z_{14} \times Z_{3}$ with group set $\{\{ x \} \times Z_{3} : x \in  Z_{14} \}$.

First, we construct an SQS(14) on $Z_{14}$  with  block set  $\mathcal{A}$. Below are the base blocks.
$$\begin{array}{lllll}
\{0,1,2,3\},&\{0,3,5,9\},&\{0,5,6,13\},&\{0,3,4,13\},&\{1,7,11,13\},\\
\{0,2,7,9\},&\{0,1,6,7\},&\{0,3,7,10\},&\{0,3,6,11\},&\{0,2,11,13\},\\
\{0,2,4,8\},&\{0,1,4,5\},& \{1,6,10,12\}.
\end{array}$$
The block set  $\mathcal{A}$ will be generated from the base blocks by (+2~mod 14).
Second, we construct a TD$(3,4,3)$ on $A\times Z_{3}$ for every block $A=\{a_{0},a_{1},a_{2},a_{3}\} \in \mathcal{A}$. By Lemma \ref{S_gdd}, there exists  a GDD$(3, 4, 42)$ of type $3^{14}$ with  block set  $\mathcal{B}$ as follows:
$$\mathcal{B}=\{\{(a_{0},x),(a_{1},y),(a_{2},z),(a_{3},u)\}: (a_{0},a_{1},a_{2},a_{3})~ {\rm and}~ (x,y,z,u)~ {\rm are~ listed~ below}\}.$$
\begin{center}
\begin{tabular}{|c|c|}
\hline
$(a_{0},a_{1},a_{2},a_{3})$&$(x,y,z,u), x, y, z, u\in Z_{3}$\\ \hline\hline

& $x + y + z + u  \equiv m$~(mod~3)\\

\raisebox{1.5ex}[0pt]{(4,~6,~8,~12),~(2,~5,~7,~11),~(8,~11,~12,~7)}&$(x, m) \in$ \{(0, 1), (1, 0), (2, 2)\}\\ \hline

&$ x + y + z + u\equiv m$~(mod~3)\\

\raisebox{1.5ex}[0pt]{(0,~2,~7,~9),~(10,~12,~3,~5)}&$(x, m) \in$ \{(0, 2), (1, 1), (2, 0)\}\\ \hline

~(8,~9,~12,~13),~(4,~6,~11,~13),~(11,~3,~7,~9)& \\

(4,~5,~11,~10),~(6,~8,~5,~3),~(10,~13,~9,~0)& $x + y + z - u \equiv$ 1~ (mod~3)\\

(12,~3,~7,~1),~(9,~4,~6,~0),~(12,~4,~9,~1)&  \\\hline

(13,~4,~10,~8),~(12,~5,~8,~1),~(12,~2,~11,~1)&\\

(12,~9,~11,~0),~(6,~12,~5,~11),~(12,~13,~3,~2)& $x + y + z - u \equiv$ 2~ (mod~3)\\

(4,~7,~9,~13),~(6,~9,~5,~10)&\\\hline

the other blocks in $\mathcal{A}$&$x + y + z + u \equiv$ 0~ (mod~3)\\\hline
\end{tabular}
\end{center}
\vspace{0.2cm}

We can check that the derived design of the GDD at every point is resolvable. The resolutions of derived designs at each point of $X$ are listed in Appendix B.
Thus, $(X, \mathcal{B})$ forms an RDGDD$(3, 4, 42)$ of type $3^{14}$.
\end{proof}

\iffalse
\vspace{1.5cm}

\begin{theorem}
$(2,1)$-SQS$(v) \Rightarrow RDGDD(3,4,vm)~with~type~m^{v}, odd m \geq 3$.
\end{theorem}

\begin{theorem}
$$
\left.
\begin{array}{r}
RDS(3,K,v)\\
\forall~k\in K, RDGDD(t,k',km)~with~type~m^{k}\\
\end{array}
\right\}\Rightarrow RDGDD(t,k',vm)~with~type~m^{v}.
$$
\end{theorem}

\begin{theorem}

$$
\left.
\begin{array}{r}
RDGDD(t,k,vg)~with~type~g^{v}\\
RDGDD(t,k^{'},kn)~with~type~n^{k}\\
\end{array}
\right\}\Rightarrow RDGDD(t,k^{'},vgn~with~type~(gn)^{v}..
$$

\end{theorem}

\begin{lemma}{\rm{\cite{2010ji}}\label{OA(3,5(6),v)}}
(1) Let $v \geq 4$ be an integer. If $v ? 2 (mod 4)$, then an $TD(3, 5, v)$ exists;

(2) Let $v$ be a positive integer which satisfies $gcd(v, 4) \neq 2$ and $gcd(v, 18) \neq 3$. Then there is an
$TD(3, 6, v)$.

\end{lemma}

\begin{lemma}{\rm{\cite{2011yin}}\label{OA(3,5,v)}}
Let $x$ be an arbitrary odd positive integer. Let $g$ be an arbitrary positive integer whose prime power factors are all $\geq 7$ such that $g \equiv 3 (mod~4)$. Then

(1) there is a $TD(3, 5, v)$ with $v = 35xg +5 \equiv 2 (mod 4)$, if $x \equiv 1 (mod 4)$;

(2) there is a $TD(3, 5, v)$ with $v = 35xg +7 \equiv 2 (mod 4)$, if $x \equiv 3 (mod 4)$.
\end{lemma}

\begin{lemma}
If there is a $TD(3,5,v)$, then there is a $RDGDD(3,4,4v)$ of type $v^{4}$.
\end{lemma}
\vspace{1.5cm}
\subsection{The construction of RDSQS}

\fi

Chang and Zhou  used RDGDDs to construct RDSQSs. Let $s = 1, m = 3$ in [4, Lemma 5.5], we have the following lemma.

\begin{lemma}{\rm{\cite{2013chang}}\label{rdsqs-dg}}
Suppose that there exists an $RDS(3,k + 1,n + 1)$. If there exist an RDSQS$(3k+1)$ and an
RDGDD$(3, 4, 3(k+1))$ of type $3^{k+1}$, then there exists an RDSQS $(3n+1)$.
\end{lemma}

\begin{lemma}\label{RDSQS(22)}
There exists an RDSQS$(22)$.
\end{lemma}
\begin{proof}
We construct an SQS$(22)$ on $Z_{21}\cup \{\infty\}$. Here, we list the base blocks under the group $Z_{21}$.
$${\small \begin{array}{lllll}
\{0, 1, 5, \infty\},&\{0,1,2,4\}, &\{0, 2, 5, 13\},&\{0, 1, 9, 18\},&\{0, 1, 11, 16\},\\
\{0, 3, 9, \infty\},&\{0,2,7,9\},&\{0, 3, 7, 16\},&\{0, 1, 6, 10\},&\{0, 2, 15, 18\},\\
\{0, 2, 10, \infty\},&\{0, 1, 7, 12\}, &\{0, 1, 13, 17\},&\{0, 2, 6, 12\},&\{0, 3, 10, 14\},\\
\{0, 7, 14, \infty\},&\{0, 1, 8, 14\},&\{0, 1, 15, 19\},&\{0, 2, 8, 11\}.
\end{array}}$$

Next, we  give the resolution of the derived design at each point of $Z_{21}\cup \{\infty\}$.  We only list the block sets of the derived designs at point $\infty$ and $0$, and the  blocks in each row form a corresponding parallel class. The  block sets of the derived designs at other points can be obtained under the action of the  group  $Z_{21}$.

Point $\infty$:
\setlength{\arraycolsep}{1.2pt}
$${\small \begin{array}{lllllll}
\{0, 1, 5\},&\{3, 4, 8\},&\{6, 7, 11\},&\{9, 10, 14\},&\{12, 13, 17\},&\{15, 16, 20\},&\{2, 18, 19\};\\
\{1, 2, 6\},&\{4, 5, 9\},&\{7, 8, 12\},&\{0, 16, 17\},&\{10, 11, 15\},&\{13, 14, 18\},&\{3, 19, 20\};\\
\{6, 8, 16\},&\{4, 7, 13\},&\{0, 2, 10\},&\{1, 12, 14\},&\{5, 17, 20\},&\{9, 11, 19\},&\{3, 15, 18\};\\

\{0, 7, 14\},&\{2, 9, 16\},&\{1, 8, 15\},&\{3, 10, 17\},&\{6, 13, 20\},&\{5, 12, 19\},&\{4, 11, 18\};\\
\{3, 6, 12\},&\{5, 7, 15\},&\{1, 9, 20\},&\{4, 16, 19\},&\{2, 14, 17\},&\{8, 10, 18\},&\{0, 11, 13\};\\
\{8, 9, 13\},&\{0, 6, 18\},&\{1, 7, 19\},&\{2, 5, 11\},&\{4, 15, 17\},&\{3, 14, 16\},&\{10, 12, 20\};\\
\{7, 9, 17\},&\{2, 8, 20\},&\{1, 3, 11\},&\{4, 6, 14\},&\{0, 12, 15\},&\{5, 16, 18\},&\{10, 13, 19\};\\
\{0, 8, 19\},&\{6, 9, 15\},&\{3, 5, 13\},&\{2, 4, 12\},&\{1, 17, 18\},&\{7, 10, 16\},&\{11, 14, 20\};\\
\{2, 3, 7\},&\{5, 6, 10\},&\{0, 4, 20\},&\{9, 12, 18\},&\{1, 13, 16\},&\{8, 11, 17\},&\{14, 15, 19\};\\
\{0, 3, 9\},&\{1, 4, 10\},&\{5, 8, 14\},&\{7, 18, 20\},&\{2, 13, 15\},&\{6, 17, 19\},&\{11, 12, 16\}.

\end{array}}$$

Point 0:
$${\small \begin{array}{lllllll}
\{1, 5, \infty\},&\{2, 7, 9\},&\{4, 15, 16\},&\{6, 13, 14\},&\{8, 17, 20\},&\{3, 11, 19\},&\{10, 12, 18\};\\
\{3, 9, \infty\},&\{2, 6, 12\},&\{5, 7, 19\},&\{1, 13, 17\},&\{4, 11, 14\},&\{8, 16, 18\},&\{10, 15, 20\};\\
\{6, 18, \infty\},&\{1, 2, 4\},&\{7, 8, 15\},&\{5, 9, 20\},&\{3, 10, 14\},&\{11, 12, 17\},&\{13, 16, 19\};\\
\{8, 19, \infty\},&\{1, 9, 18\},&\{3, 4, 12\},&\{7, 13, 20\},&\{2, 14, 16\},&\{5, 10, 11\},&\{6, 15, 17\};\\
\{7, 14, \infty\},&\{2, 3, 17\},&\{5, 8, 12\},&\{1, 15, 19\},&\{4, 13, 18\},&\{9, 10, 16\},&\{6, 11, 20\};\\
\{4, 20, \infty\},&\{1, 6, 10\},&\{2, 8, 11\},&\{3, 7, 16\},&\{9, 12, 13\},&\{5, 14, 15\},&\{17, 18, 19\};\\
\{2, 10, \infty\},&\{4, 5, 17\},&\{1, 8, 14\},&\{6, 9, 19\},&\{7, 11, 18\},&\{3, 13, 15\},&\{12, 16, 20\};\\
\{12, 15, \infty\},&\{4, 6, 7\},&\{3, 5, 18\},&\{1, 11, 16\},&\{8, 10, 13\},&\{2, 19, 20\},&\{9, 14, 17\};\\
\{11, 13, \infty\},&\{4, 8, 9\},&\{5, 6, 16\},&\{1, 3, 20\},&\{2, 15, 18\},&\{7, 10, 17\},&\{12, 14, 19\};\\
\{16, 17, \infty\},&\{3, 6, 8\},&\{1, 7, 12\},&\{2, 5, 13\},&\{4, 10, 19\},&\{9, 11, 15\},&\{14, 18, 20\}.
\end{array}}$$
\end{proof}

\begin{lemma}{\rm{\cite{2013chang}}\label{rds(qn+1)}}
There exists an $RDS(3, q+1, q^{n} +1)$ for any prime power $q$ and any positive integer $n$.
\end{lemma}

\begin{theorem}\label{rdsqs}
There exists  an RDSQS$(3\times 7^{n}+1)$ and an RDSQS$(3\times 13^{n}+1)$ for   any nonnegative integer $n$.
\end{theorem}

\begin{proof} Employ Lemma \ref{rds(qn+1)} with $q = 7$; both an RDSQS$(22)$ and an RDGDD$(3, 4, 24)$ of type $3^{8}$ exist from Lemma \ref{RDSQS(22)} and \ref{gdd3_8} respectively,  then an RDSQS$(3\times 7^{n}+1)$ exists by Lemma \ref{rdsqs-dg}. Similarly, employ Lemma \ref{rds(qn+1)} with $q = 13$;  an RDSQS$(40)$ exists from \cite{2009yuan} and an RDGDD$(3, 4, 42)$ of type $3^{14}$ exists from  Lemma \ref{gdd3_14}, then an RDSQS$(3\times 13^{n}+1)$ exists by Lemma \ref{rdsqs-dg}.
\end{proof}

\section{A construction for RDSQSs}

In this section, we first introduce a special combinatorial structure RDSQS$^{\ast}(v)$, which weakens the properties of (2, 1)-RSQS$^{\ast}(v)$ in \cite{2021xu}. We will use it to present a  recursive construction for RDSQSs.

Let  $(X,\mathcal{B})$ be an SQS$(v)$ with $v \equiv 1~({\rm mod}~3)$. $(X,\mathcal{B})$ is denoted by RDSQS$^{\ast}(v)$ if it satisfies the following properties:

$(1)$~for any $x \in X$,  there is a parallel class $P_{x}^{'}$ of $X\setminus \{x\}$ in $\mathcal{B}_{x} = \{B \setminus \{x\} : x \in B, B \in \mathcal{B} \}$, and set  $P_{x}^{'} = \{B_{x,k}: B_{x,k}\in \mathcal{B}_{x}, 1\leq k \leq \frac{v-1}{3}\}$;

$(2)$ the multiset $\mathcal{M}$ containing each triple of $P_{x}^{'}$ three times and  each triple of $\mathcal{B}_{x} \setminus P_{x}^{'}$ twice can be partitioned into $v-1$ parallel classes $P_{x,k}^{l}$, $1\leq k \leq \frac{v-1}{3}$, $1\leq l \leq 3$;

$(3)$~for $1\leq k \leq \frac{v-1}{3}$, three parallel classes~$P_{x,k}^{l}$$(1\leq l \leq 3)$ share a common triple $B_{x,k}$.

\begin{lemma}\label{rdsqs28}
There exists an RDSQS$^{\ast}(28)$.
\end{lemma}
\begin{proof}
An SQS$(28)$ was given by Ji and Zhu \cite{2003ji}. The design is constructed on  $X =Z_{7} \times Z_{4}$ and the base blocks modulo $(7, -)$ are listed in Appendix C.

In fact, the SQS(28) is also an RDSQS$^{\ast}(28)$. We need to prove that it satisfies the properties $(1),(2),$ $(3)$. We list the resolutions of the derived designs at the points $(0, 0), (0, 1), (0, 2), (0, 3)$ in Appendix C, and  the resolutions of the derived designs at the other points of $X$ can be obtained modulo $(7, -)$.
\end{proof}

A GDD$(t, k ,v)$ (resp. S$(t, k, v)$) is called $i$-{\it resolvable}, $0 < i < t$, if its block set can be partitioned into some disjoint GDD$(i, k ,v)$s (resp. S$(i, k, v)$s). A 1-resolvable GDD$(t, k ,v)$ (resp. S$(t, k, v)$) is just a resolvable GDD$(t, k ,v)$ (resp. S$(t, k, v)$). It is obvious that a 2-resolvable GDD$(t, k ,v)$ (resp. S$(t, k, v)$) must be a RDGDD$(t, k, v)$ (resp. RDS$(t, k, v)$).

\begin{example}\label{rdsqs16}
There exists a $2$-resolvable SQS$(16)$ (resp. RDSQS$(16))$.
\end{example}
\begin{proof}
This example comes from \cite{1976baker}. Here we exhibit it on the finite field GF$(16)$. Let $\alpha$ be a primitive element of GF$(16)$, where $\alpha$ is the root of the primitive polynomial $f(x) = x^{4} + x + 1$ over GF$(2)$. Let $X =$ GF$(16)$ and $\mathcal{B}$ consists of all quadruples $\{x_{1}, x_{2}, x_{3}, x_{4}\}$ of $X$ with $x_{1} + x_{2} + x_{3} + x_{4} = 0$. It is well-known that $(X, \mathcal{B})$ is a Boolean SQS$(16)$. It is easy to see that each quadruple in $\mathcal{B}$ generates a parallel class under the action of the additive group (GF$(16), +)$.\vspace{0.2cm}

We list the following base blocks, which can generate a $2$-resolvable SQS$(16)$ (also an RDSQS$(16)$), and  the blocks in each row form a resolvable $S(2, 4, 16)$, under the action of the additive group (GF$(16), +)$.
\setlength{\arraycolsep}{1.6pt}
$$\begin{array}{lllllll}
\vspace{0.05cm}
\uwave{\{0,1,\alpha,\alpha^{4}\}},&\underline{\{0,\alpha^{2},\alpha^{3},\alpha^{6}\}},&\underline{\{0,\alpha^{5},\alpha^{7},\alpha^{13}\}},&\underline{\{0,\alpha^{10},\alpha^{11},\alpha^{14}\}},&\underline{\{0,\alpha^{8},\alpha^{9},\alpha^{12}\}};\\\vspace{0.1cm}
\{0,1,\alpha^{2},\alpha^{8}\},&\{0,\alpha,\alpha^{7},\alpha^{14}\},&\{0,\alpha^{4},\alpha^{6},\alpha^{12}\},&\underline{\{0,\alpha^{3},\alpha^{5},\alpha^{11}\}},&\underline{\{0,\alpha^{9},\alpha^{10},\alpha^{13}\}};\\\vspace{0.2cm}
\{0,1,\alpha^{3},\alpha^{14}\},&\{0,\alpha^{2},\alpha^{4},\alpha^{10}\},&\{0,\alpha,\alpha^{12},\alpha^{13}\},&\underline{\{0,\alpha^{5},\alpha^{6},\alpha^{9}\}},&\underline{\{0,\alpha^{7},\alpha^{8},\alpha^{11}\}};\\\vspace{0.2cm}
\{0,1,\alpha^{5},\alpha^{10}\},&\{0,\alpha,\alpha^{6},\alpha^{11}\},&\{0,\alpha^{4},\alpha^{9},\alpha^{14}\},&\underline{\{0,\alpha^{2},\alpha^{7},\alpha^{12}\}},&\underline{\{0,\alpha^{3},\alpha^{8},\alpha^{13}\}};\\\vspace{0.2cm}
\{0,1,\alpha^{6},\alpha^{13}\},&\{0,\alpha^{3},\alpha^{4},\alpha^{7}\},&\{0,\alpha,\alpha^{8},\alpha^{10}\},&\underline{\{0,\alpha^{2},\alpha^{9},\alpha^{11}\}},&\underline{\{0,\alpha^{5},\alpha^{12},\alpha^{14}\}};\\\vspace{0.2cm}
\{0,1,\alpha^{11},\alpha^{12}\},&\{0,\alpha,\alpha^{3},\alpha^{9}\},&\{0,\alpha^{4},\alpha^{5},\alpha^{8}\},&\underline{\{0,\alpha^{6},\alpha^{7},\alpha^{10}\}},&\underline{\{0,\alpha^{2},\alpha^{13},\alpha^{14}\}};\\
\{0,1,\alpha^{7},\alpha^{9}\},&\{0,\alpha,\alpha^{2},\alpha^{5}\},&\{0,\alpha^{4},\alpha^{11},\alpha^{13}\},&\underline{\{0,\alpha^{6},\alpha^{8},\alpha^{14}\}},&\underline{\{0,\alpha^{3},\alpha^{10},\alpha^{12}\}}.
\end{array}$$

In addition, this design also has the following properties.

(1)~Under the action of the additive group (GF$(16), +)$, the base block $\{0,1,\alpha,\alpha^{4}\}$  generates a parallel class $P = \{\{0,1,\alpha,\alpha^{4}\}$,$\{\alpha^{2},\alpha^{8},\alpha^{5},\alpha^{10}\}$,$\{\alpha^{3},\alpha^{14},$ $\alpha^{9},\alpha^{7}\}$,
$\{\alpha^{6},\alpha^{13},$
$\alpha^{11},\alpha^{12}\}\}$. With $P$ as groups, the underlined base blocks can generate a 2-resolvable TD$(3,4,4)$ (also an RDTD$(3,4,4)$). We list its resolution below,   the blocks in each row form a resolvable TD$(2, 4, 4)$ under the action of the additive group (GF$(16), +)$.\vspace{-0.2cm}
$$\begin{array}{lllllll}
\{0,\alpha^{2},\alpha^{3},\alpha^{6}\},&\{0,\alpha^{5},\alpha^{7},\alpha^{13}\},&\{0,\alpha^{10},\alpha^{11},\alpha^{14}\},&\{0,\alpha^{8},\alpha^{9},\alpha^{12}\};\\
\{0,\alpha^{3},\alpha^{5},\alpha^{11}\},&\{0,\alpha^{2},\alpha^{7},\alpha^{12}\},&\{0,\alpha^{9},\alpha^{10},\alpha^{13}\},&\{0,\alpha^{6},\alpha^{8},\alpha^{14}\};\\
\{0,\alpha^{5},\alpha^{6},\alpha^{9}\},&\{0,\alpha^{7},\alpha^{8},\alpha^{11}\},&\{0,\alpha^{2},\alpha^{13},\alpha^{14}\},&\{0,\alpha^{3},\alpha^{10},\alpha^{12}\};\\
\{0,\alpha^{3},\alpha^{8},\alpha^{13}\},&\{0,\alpha^{2},\alpha^{9},\alpha^{11}\},&\{0,\alpha^{5},\alpha^{12},\alpha^{14}\},&\{0,\alpha^{6},\alpha^{7},\alpha^{10}\}.
\end{array}$$

(2)~We will give a one-to-one correspondence between the elements of GF$(16)$ and $Z_{4} \times Z_{4}$ to rename the four groups in the above TD(3, 4, 4) as $\{x\}\times Z_{4}, x \in Z_{4}$, where $(a, i)$ is denoted by $a_{i}$:
\begin{center}
$(0,1,\alpha,\alpha^{4})\triangleq (0_{0},0_{1},0_{2},0_{3})$,~~~~~~~~
$(\alpha^{2},\alpha^{8},\alpha^{5},\alpha^{10})\triangleq (1_{0},1_{1},1_{2},1_{3})$,

$(\alpha^{3},\alpha^{14},\alpha^{9},\alpha^{7})\triangleq (2_{0},2_{1},2_{2},2_{3})$,~~~~
$(\alpha^{6},\alpha^{13},\alpha^{11},\alpha^{12})\triangleq (3_{0},3_{1},3_{2},3_{3})$.
\end{center}
There exists a one-factorization on $Z_{4}$, $\mathcal{F}=\{F_{1}, F_{2}, F_{3}\}$, where
\begin{center}
$F_1 = \{\{0,1\}, \{2,3\}\},$ ~~
$F_2 = \{\{0,2\}, \{1,3\}\},$~~
$F_3 = \{\{0,3\}, \{1,2\}\}.$
\end{center}\vspace{-0.5cm}
Define

\centerline{ $ \mathcal{C} = \{ \{ (x,a),(x,b),(y,c),(y,d) \}: x,y \in Z_{4}, x\neq y, \{a,b\},\{c,d\} \in F_{s}, 1\leq s \leq 3\}.$}

\noindent
From one-to-one correspondence, $ \mathcal{C}$  happens to be the block set formed by the base blocks not underlined under the action of the additive group (GF$(16), +)$.

\end{proof}

\begin{construction}\label{RDSQS4v}
If there is an RDSQS$^{\ast}(v)$, then there is an RDSQS$(4v)$.
\end{construction}
\begin{proof}
Let $(X,\mathcal{B})$ be an RDSQS$^{\ast}(v)$. For any $x \in X$, there is a parallel class $P_{x}^{'}$ of $X\setminus \{x\}$ in $\mathcal{B}_{x} = \{B \setminus \{x\} : x \in B, B \in \mathcal{B} \}$, let   $P_{x}^{'} = \{B_{x,k}: B_{x,k}\in \mathcal{B}_{x}, 1\leq k \leq \frac{v-1}{3}\}$; define the multiset $\mathcal{M}$ containing each triple of $P_{x}^{'}$ three times, while each triple of $\mathcal{B}_{x} \setminus P_{x}^{'}$ twice,  the blocks of $\mathcal{M}$  can be partitioned into $v-1$ parallel classes $P_{x,k}^{l}$, $1\leq k \leq \frac{v-1}{3}$, $1\leq l \leq 3$; and for $1\leq k \leq \frac{v-1}{3}$, three parallel classes~$P_{x,k}^{l}$$(1\leq l \leq 3)$ share a common triple $B_{x,k}$. Let $X^{'} = X \times Z_{4}$.

For each block $B \in \mathcal{B}$, we construct an RDTD$(3,4,4)$ ($B \times Z_{4}, \mathcal{G}_{B} ,\mathcal{A}^{B})$  by the construction in Example \ref{rdsqs16},  where $\mathcal{G}_{B} = \{ \{x\} \times Z_{4}: x \in B\}$. For every point $(x,i)\in B \times Z_{4}$, denote the corresponding  derived design as $((B\setminus \{x\}) \times Z_{4}, \mathcal{A}_{(x,i)}^{B})$, and $\mathcal{A}_{(x,i)}^{B}$  can be partitioned into 4  parallel classes $\mathcal{A}_{(x,i)}^{B}(j),~j \in Z_{4}$ on $(B\setminus \{x\})\times Z_{4}$.

For any subset $Y\subseteq X$, define a block set
$$\mathcal{C}^{Y} = \{ \{  (x,a),(x,b),(y,c),(y,d)\}: \{x,y\} \subseteq Y, x \neq y, \{a,b\},\{c,d\} \in F_{s}, 1\leq s \leq 3\},$$
where $\{F_{1}, F_{2}, F_{3}\}$ is a one-factorization on $Z_{4}$ defined in Example \ref{rdsqs16}.

Define
\begin{center}
$\mathcal{D} = (\bigcup\limits_{B \in \mathcal{B}}\mathcal{A}^{B})~\bigcup~\mathcal{C}^{X}~\bigcup~\{ \{x\} \times Z_{4}: x \in X\}$.
\end{center}

\hspace{-0.6cm}Then it is routine to check that $(X^{'}, \mathcal{D})$ forms an SQS$(4v)$. Next, we prove that $(X^{'}, \mathcal{D})$ is an RDSQS$(4v)$.

For any subset $Y\subseteq X$, every point $(x,i)\in Y \times Z_{4}$, let $((Y \times Z_{4})\setminus \{(x,i)\}, \mathcal{C}_{(x,i)}^{Y})$ be the derived design of $\mathcal{C}^{Y}$ at the point $(x,i)$.

For any  $(x,i)\in X^{'}$, let
\begin{center}
$\mathcal{D}_{(x,i)}$ = $(\bigcup\limits_{x\in B,B \in \mathcal{B}}\mathcal{A}_{(x,i)}^{B})~\bigcup~\mathcal{C}_{(x,i)}^{X}~\bigcup~\{ (\{x\} \times Z_{4}) \setminus \{(x,i)\}\}$
\end{center}
\hspace{3cm} $=\mathop{(\bigcup\limits_{x\in B, B \in \mathcal{B}\atop B \setminus \{x\}\notin P_{x}^{'}} } \mathcal{A}_{(x,i)}^{B})~\bigcup~\mathop{(\bigcup\limits_{x\in B, B \in \mathcal{B}\atop B \setminus \{x\}\in P_{x}^{'}} } \mathcal{A}_{(x,i)}^{B})~\bigcup$

\vspace{0.2cm}

\hspace{2.8cm}$~\mathop{(\bigcup\limits_{x\in B, B \in \mathcal{B}\atop B \setminus \{x\}\in P_{x}^{'}} }\mathcal{C}_{(x,i)}^{B})~\bigcup~\{ (\{x\} \times Z_{4}) \setminus \{(x,i)\}\}$.\vspace{0.2cm}

\hspace{-0.6cm}Then $((X \times Z_{4})\setminus \{(x,i)\}, \mathcal{D}_{(x,i)})$ is the derived design  at point $(x,i)\in X^{'}$. \vspace{0.2cm}

For each block $B_{x,k}\in P_{x}^{'}$, $1\leq k \leq \frac{v-1}{3}$, let
$$\mathcal{E}^{B_{x,k}}=\mathcal{A}_{(x,i)}^{B_{x,k}}~\cup~\mathcal{C}_{(x,i)}^{B_{x,k}}
 \cup~\{(\{x\} \times Z_{4}) \setminus \{(x,i)\}\}.$$
 From the construction of the  RDSQS$(16)$ in Example \ref{rdsqs16},  $\mathcal{E}^{B_{x,k}}$ is the block sets of the derived design at point $(x,i)$, which can be partitioned into 7 parallel  classes $\mathcal{E}_{(x,i)}^{B_{x,k}}(j), $ $j\in Z_7$ on $((B_{x,k}\cup \{x\})\times Z_{4})\backslash \{(x,i)\}$, where $~((\{x\} \times Z_{4}) \setminus \{(x,i)\}) \subseteq \mathcal{E}_{(x,i)}^{B_{x,k}}(6)$.\vspace{0.2cm}

For $1\leq k \leq \frac{v-1}{3}$, $1 \leq l \leq 3$, and $r=0,1$, let
\begin{center}
$\mathcal{D}_{(x,i)}(k,l,r) = $~$\mathop{( \bigcup\limits_{ x \in B,~ B\in \mathcal{B} \atop B \setminus \{x\} \in P_{x,k}^{l} \setminus \{B_{x,k}\}}}$ $ \mathcal{A}_{(x,i)}^{B}(r^{'}) )~\bigcup ~\mathcal{E}_{(x,i)}^{B_{x,k}}(2(l-1)+r)$,
\end{center}

\noindent where the value $r^{'}$ is equal to $r$ or $r + 2$ according to the first or the second occurrence of $B_{x,k}$. Note that the block $B_{x,k}\in \mathcal{B}_{x}\backslash  P^{'}_{x}$ occurs twice in the multiset $\mathcal{M}$ from the definition of RDSQS$^{\ast}(v)$.

Let
$$\mathcal{D}_{(x,i)}(2v-1) = ( \bigcup\limits_{ B_{x,k} \in P^{'}_{x},1 \leq k \leq \frac{v-1}{3}}  (\mathcal{E}_{(x,i)}^{B_{x,k}}(6)\setminus \{(\{x\} \times Z_{4})\setminus \{(x,i)\}\})$$
\hspace{3.9cm} $\bigcup~ \{(\{x\} \times Z_{4})\setminus \{(x,i)\}\}$.

It is not difficult to check that each  $\mathcal{D}_{(x,i)}(k,l,r)$, $1\leq k \leq \frac{v-1}{3}$, $1 \leq l \leq 3$, $r=0,1$, and $\mathcal{D}_{(x,i)}(2v-1)$
is a parallel class on $(X \times Z_{4})\setminus \{(x,i)\}$. So, the derived design   $((X \times Z_{4})\setminus \{(x,i)\}, \mathcal{D}_{(x,i)})$ is  resolvable. By the arbitrariness of the point $(x,i)\in X^{'}$, we know that $(X^{'}, \mathcal{D})$ is an RDSQS$(4v)$.
\end{proof}

\begin{theorem}\label{rdsqs1}
There exists an RDSQS$(7\times 4^{n})$ for  any positive integer $n$.
\end{theorem}
\begin{proof} There exists an RDSQS$^{\ast}(28)$ by Lemma \ref{rdsqs28}, then there exists an RDSQS$(7\times 4^{n})$  by applying Construction \ref{RDSQS4v} repeatedly.
\end{proof}

\section{Main results}
In this section, we extend our new results to a wider range by combining the known constructions for RDSQSs and LKTSs.

\begin{lemma}{\rm{\cite{2010yuan}\label{rds(uv+1)}}}
If there exists an RDSQS$(u + 1)$ and an RDSQS$(v + 1)$, then there exists an RDSQS$(uv + 1)$.
\end{lemma}

\begin{lemma}{\rm{\cite{2014zhou}\label{zhourdsqs}}}
Let $L = \{4^{m}25^{n} - 1 : m \geq 1, n \geq 0\}~\cup$ $\{4 \times 7^{n} - 1 : n \geq 0\}$ $\cup$ $\{2q^{n} + 1 : n \geq 0, q \equiv 7~({\rm mod}~12)$ and $q$ is a prime power$\}$. Then there exist an RDSQS$(v + 1)$ and an LKTS$(3v)$ where $v$ is the product of some elements in $L$.
\end{lemma}

\begin{lemma}{\rm{\cite{2019chang}\label{changrdsqs}}}
There exists an RDSQS$(2v)$ for $v \in \{23, 29, 41, 47, 53, 59, 71, 83\}$.
\end{lemma}

\begin{theorem}
Let $L = \{7\times 4^{n} - 1: n\geq 1\}~\cup~\{3\times 7^{n}: n\geq 0\}~\cup~\{3\times 13^{n}: n\geq 0\}~\cup~\{4^{m}25^{n} - 1 : m \geq 1, n \geq 0\} ~\cup~ \{4 \times 7^{n} - 1 : n \geq 0\} ~\cup~ \{2q^{n} + 1 : n \geq 0, q \equiv 7~({\rm mod~12})$ and $q$ is a prime power$\}~\cup~ \{45,57,81,93,105,117,141,165\}$. Then there exists an RDSQS$(v + 1)$ and an LKTS$(3v)$ where $v$ is the product of some elements in $L$.
\end{theorem}
\begin{proof}
Applying Theorem \ref{rdsqs}, Theorem \ref{rdsqs1},  Lemma \ref{zhourdsqs} and Lemma \ref{changrdsqs} to the known product constructions in Lemma \ref{rds(uv+1)}, we obtain an RDSQS$(v+1)$s. By Theorem \ref{1lkts}, there exists an LKTS$(3v)$.
\end{proof}

%In this paper, we introduced a special combinatorial structure RDSQS$^{\ast}(v)$, which weakens the properties of (2, 1)-RSQS$^{\ast}(v)$ in \cite{2021xu}. We used it to present a construction for RDSQSs. By Theorem \ref{1lkts}, RDSQSs could produce LKTSs. So it is worth finding more constructions of RDSQSs and  RDSQS$^{\ast}$s.

%\end{document}

\newpage

\appendix
\renewcommand{\appendixname}{Appendix~\Alph{section}}
\section{Appendix: The derived designs in Lemma \ref{gdd3_8}}
For the given GDD$(3, 4, 24)$ of type $3^{8}$  in Lemma \ref{gdd3_8}, we list all the blocks of the derived design GDD$(2, 3, 21)$ of type $3^{7}$ at each point $x \in ( \{ \infty \} \cup Z_{7} ) \times Z_{3}$. The seven blocks in each row form a parallel class corresponding to the derived design, where $(i,j)$ is denoted by $i_{j}$.

point $\infty_{0}$:
$$
\begin{array}{llllllll}
0_{0}1_{0}3_{0}&1_{1}2_{0}4_{2}&2_{1}3_{1}5_{1}&3_{2}4_{0}6_{1}&0_{2}4_{1}5_{0}&1_{2}5_{2}6_{2}&0_{1}2_{2}6_{0}\\

0_{0}1_{1}3_{2}&1_{0}2_{1}4_{2}&2_{0}3_{1}5_{2}&3_{0}4_{0}6_{0}&0_{1}4_{1}5_{1}&1_{2}5_{0}6_{1}&0_{2}2_{2}6_{2}\\

0_{0}1_{2}3_{1}&1_{0}2_{2}4_{1}&2_{1}3_{2}5_{0}&3_{0}4_{2}6_{1}&0_{2}4_{0}5_{1}&1_{1}5_{2}6_{0}&0_{1}2_{0}6_{2}\\

0_{1}1_{0}3_{2}&1_{2}2_{1}4_{0}&2_{2}3_{0}5_{1}&3_{1}4_{2}6_{0}&0_{0}4_{1}5_{2}&1_{1}5_{0}6_{2}&0_{2}2_{0}6_{1}\\

0_{1}1_{1}3_{1}&1_{2}2_{2}4_{2}&2_{0}3_{2}5_{1}&3_{0}4_{1}6_{2}&0_{0}4_{0}5_{0}&1_{0}5_{2}6_{1}&0_{2}2_{1}6_{0}\\

0_{1}1_{2}3_{0}&1_{0}2_{0}4_{0}&2_{2}3_{1}5_{0}&3_{2}4_{1}6_{0}&0_{2}4_{2}5_{2}&1_{1}5_{1}6_{1}&0_{0}2_{1}6_{2}\\

0_{2}1_{0}3_{1}&1_{1}2_{1}4_{1}&2_{0}3_{0}5_{0}&3_{2}4_{2}6_{2}&0_{1}4_{0}5_{2}&1_{2}5_{1}6_{0}&0_{0}2_{2}6_{1}\\

0_{2}1_{1}3_{0}&1_{2}2_{0}4_{1}&2_{2}3_{2}5_{2}&3_{1}4_{0}6_{2}&0_{0}4_{2}5_{1}&1_{0}5_{0}6_{0}&0_{1}2_{1}6_{1}\\

0_{2}1_{2}3_{2}&1_{1}2_{2}4_{0}&2_{1}3_{0}5_{2}&3_{1}4_{1}6_{1}&0_{1}4_{2}5_{0}&1_{0}5_{1}6_{2}&0_{0}2_{0}6_{0}
\end{array}$$

point $\infty_{1}$:
$$\begin{array}{llllllll}
0_{2}1_{0}3_{2}&1_{1}2_{2}4_{1}&2_{1}3_{1}5_{0}&3_{0}4_{2}6_{0}&0_{0}4_{0}5_{2}&1_{2}5_{1}6_{2}&0_{1}2_{0}6_{1}\\

0_{1}1_{0}3_{0}&1_{2}2_{0}4_{2}&2_{2}3_{1}5_{2}&3_{2}4_{0}6_{0}&0_{0}4_{1}5_{1}&1_{1}5_{0}6_{1}&0_{2}2_{1}6_{2}\\

0_{2}1_{1}3_{1}&1_{0}2_{0}4_{1}&2_{2}3_{2}5_{1}&3_{0}4_{0}6_{2}&0_{0}4_{2}5_{0}&1_{2}5_{2}6_{1}&0_{1}2_{1}6_{0}\\

0_{1}1_{1}3_{2}&1_{2}2_{1}4_{1}&2_{0}3_{0}5_{2}&3_{1}4_{0}6_{1}&0_{2}4_{2}5_{1}&1_{0}5_{0}6_{2}&0_{0}2_{2}6_{0}\\

0_{0}1_{2}3_{2}&1_{1}2_{1}4_{2}&2_{0}3_{1}5_{1}&3_{0}4_{1}6_{1}&0_{2}4_{0}5_{0}&1_{0}5_{2}6_{0}&0_{1}2_{2}6_{2}\\

0_{0}1_{1}3_{0}&1_{2}2_{2}4_{0}&2_{1}3_{2}5_{2}&3_{1}4_{2}6_{2}&0_{1}4_{1}5_{0}&1_{0}5_{1}6_{1}&0_{2}2_{0}6_{0}\\

0_{2}1_{2}3_{0}&1_{0}2_{2}4_{2}&2_{0}3_{2}5_{0}&3_{1}4_{1}6_{0}&0_{1}4_{0}5_{1}&1_{1}5_{2}6_{2}&0_{0}2_{1}6_{1}\\

0_{1}1_{2}3_{1}&1_{0}2_{1}4_{0}&2_{2}3_{0}5_{0}&3_{2}4_{2}6_{1}&0_{2}4_{1}5_{2}&1_{1}5_{1}6_{0}&0_{0}2_{0}6_{2}\\

0_{0}1_{0}3_{1}&1_{1}2_{0}4_{0}&2_{1}3_{0}5_{1}&3_{2}4_{1}6_{2}&0_{1}4_{2}5_{2}&1_{2}5_{0}6_{0}&0_{2}2_{2}6_{1}
\end{array}$$

point $\infty_{2}$:
$$\begin{array}{llllllll}
0_{2}1_{2}3_{1}&1_{0}2_{0}4_{2}&2_{2}3_{2}5_{0}&3_{0}4_{1}6_{0}&0_{0}4_{0}5_{1}&1_{1}5_{2}6_{1}&0_{1}2_{1}6_{2}\\

0_{2}1_{0}3_{0}&1_{2}2_{0}4_{0}&2_{1}3_{1}5_{2}&3_{2}4_{2}6_{0}&0_{0}4_{1}5_{0}&1_{1}5_{1}6_{2}&0_{1}2_{2}6_{1}\\

0_{1}1_{0}3_{1}&1_{1}2_{0}4_{1}&2_{1}3_{2}5_{1}&3_{0}4_{0}6_{1}&0_{0}4_{2}5_{2}&1_{2}5_{0}6_{2}&0_{2}2_{2}6_{0}\\

0_{1}1_{2}3_{2}&1_{0}2_{1}4_{1}&2_{2}3_{1}5_{1}&3_{0}4_{2}6_{2}&0_{2}4_{0}5_{2}&1_{1}5_{0}6_{0}&0_{0}2_{0}6_{1}\\

0_{2}1_{1}3_{2}&1_{2}2_{2}4_{1}&2_{0}3_{0}5_{1}&3_{1}4_{2}6_{1}&0_{1}4_{0}5_{0}&1_{0}5_{2}6_{2}&0_{0}2_{1}6_{0}\\

0_{0}1_{0}3_{2}&1_{1}2_{1}4_{0}&2_{2}3_{0}5_{2}&3_{1}4_{1}6_{2}&0_{2}4_{2}5_{0}&1_{2}5_{1}6_{1}&0_{1}2_{0}6_{0}\\

0_{0}1_{2}3_{0}&1_{1}2_{2}4_{2}&2_{0}3_{1}5_{0}&3_{2}4_{0}6_{2}&0_{1}4_{1}5_{2}&1_{0}5_{1}6_{0}&0_{2}2_{1}6_{1}\\

0_{1}1_{1}3_{0}&1_{2}2_{1}4_{2}&2_{0}3_{2}5_{2}&3_{1}4_{0}6_{0}&0_{2}4_{1}5_{1}&1_{0}5_{0}6_{1}&0_{0}2_{2}6_{2}\\

0_{0}1_{1}3_{1}&1_{0}2_{2}4_{0}&2_{1}3_{0}5_{0}&3_{2}4_{1}6_{1}&0_{1}4_{2}5_{1}&1_{2}5_{2}6_{0}&0_{2}2_{0}6_{2}
\end{array}$$

point $0_{0}$:
$$\begin{array}{llllllll}
1_{0}3_{0}\infty_{0}&4_{0}5_{2}\infty_{1}&2_{1}6_{0}\infty_{2}&1_{2}2_{2}5_{0}&2_{0}3_{2}4_{2}&3_{1}5_{1}6_{1}&1_{1}4_{1}6_{2}\\

1_{0}3_{1}\infty_{1}&4_{1}5_{2}\infty_{0}&2_{2}6_{2}\infty_{2}&1_{2}2_{1}5_{1}&2_{0}3_{0}4_{0}&3_{2}5_{0}6_{1}&1_{1}4_{2}6_{0}\\

1_{0}3_{2}\infty_{2}&4_{2}5_{1}\infty_{0}&2_{2}6_{0}\infty_{1}&1_{1}2_{1}5_{0}&2_{0}3_{1}4_{1}&3_{0}5_{2}6_{1}&1_{2}4_{0}6_{2}\\

1_{2}3_{1}\infty_{0}&4_{1}5_{0}\infty_{2}&2_{1}6_{1}\infty_{1}&1_{1}2_{0}5_{1}&2_{2}3_{0}4_{2}&3_{2}5_{2}6_{2}&1_{0}4_{0}6_{0}\\

1_{1}3_{1}\infty_{2}&4_{1}5_{1}\infty_{1}&2_{2}6_{1}\infty_{0}&1_{2}2_{0}5_{2}&2_{1}3_{2}4_{0}&3_{0}5_{0}6_{0}&1_{0}4_{2}6_{2}\\
\end{array}$$
$$\begin{array}{llllllll}

1_{2}3_{0}\infty_{2}&4_{0}5_{0}\infty_{0}&2_{0}6_{2}\infty_{1}&1_{1}2_{2}5_{2}&2_{1}3_{1}4_{2}&3_{2}5_{1}6_{0}&1_{0}4_{1}6_{1}\\

1_{1}3_{0}\infty_{1}&4_{0}5_{1}\infty_{2}&2_{1}6_{2}\infty_{0}&1_{0}2_{0}5_{0}&2_{2}3_{2}4_{1}&3_{1}5_{2}6_{0}&1_{2}4_{2}6_{1}\\

1_{1}3_{2}\infty_{0}&4_{2}5_{0}\infty_{1}&2_{0}6_{1}\infty_{2}&1_{0}2_{1}5_{2}&2_{2}3_{1}4_{0}&3_{0}5_{1}6_{2}&1_{2}4_{1}6_{0}\\

1_{2}3_{2}\infty_{1}&4_{2}5_{2}\infty_{2}&2_{0}6_{0}\infty_{0}&1_{0}2_{2}5_{1}&2_{1}3_{0}4_{1}&3_{1}5_{0}6_{2}&1_{1}4_{0}6_{1}
\end{array}$$

point $0_{1}$:
$$\begin{array}{llllllll}
1_{1}3_{0}\infty_{2}&4_{2}5_{2}\infty_{1}&2_{2}6_{0}\infty_{0}&1_{0}2_{0}5_{1}&2_{1}3_{1}4_{1}&3_{2}5_{0}6_{2}&1_{2}4_{0}6_{1}\\

1_{2}3_{0}\infty_{0}&4_{2}5_{1}\infty_{2}&2_{2}6_{2}\infty_{1}&1_{0}2_{1}5_{0}&2_{0}3_{2}4_{1}&3_{1}5_{2}6_{1}&1_{1}4_{0}6_{0}\\

1_{1}3_{1}\infty_{0}&4_{0}5_{0}\infty_{2}&2_{0}6_{1}\infty_{1}&1_{0}2_{2}5_{2}&2_{1}3_{2}4_{2}&3_{0}5_{1}6_{0}&1_{2}4_{1}6_{2}\\

1_{1}3_{2}\infty_{1}&4_{1}5_{2}\infty_{2}&2_{0}6_{2}\infty_{0}&1_{2}2_{2}5_{1}&2_{1}3_{0}4_{0}&3_{1}5_{0}6_{0}&1_{0}4_{2}6_{1}\\

1_{0}3_{2}\infty_{0}&4_{1}5_{0}\infty_{1}&2_{2}6_{1}\infty_{2}&1_{1}2_{1}5_{1}&2_{0}3_{1}4_{0}&3_{0}5_{2}6_{2}&1_{2}4_{2}6_{0}\\

1_{2}3_{1}\infty_{1}&4_{0}5_{2}\infty_{0}&2_{1}6_{2}\infty_{2}&1_{1}2_{2}5_{0}&2_{0}3_{0}4_{2}&3_{2}5_{1}6_{1}&1_{0}4_{1}6_{0}\\

1_{0}3_{0}\infty_{1}&4_{2}5_{0}\infty_{0}&2_{0}6_{0}\infty_{2}&1_{2}2_{1}5_{2}&2_{2}3_{2}4_{0}&3_{1}5_{1}6_{2}&1_{1}4_{1}6_{1}\\

1_{0}3_{1}\infty_{2}&4_{0}5_{1}\infty_{1}&2_{1}6_{1}\infty_{0}&1_{2}2_{0}5_{0}&2_{2}3_{0}4_{1}&3_{2}5_{2}6_{0}&1_{1}4_{2}6_{2}\\

1_{2}3_{2}\infty_{2}&4_{1}5_{1}\infty_{0}&2_{1}6_{0}\infty_{1}&1_{1}2_{0}5_{2}&2_{2}3_{1}4_{2}&3_{0}5_{0}6_{1}&1_{0}4_{0}6_{2}
\end{array}$$

point $0_{2}$:
$$\begin{array}{llllllll}
1_{2}3_{0}\infty_{1}&4_{1}5_{1}\infty_{2}&2_{1}6_{0}\infty_{0}&1_{0}2_{0}5_{2}&2_{2}3_{2}4_{2}&3_{1}5_{0}6_{1}&1_{1}4_{0}6_{2}\\

1_{2}3_{2}\infty_{0}&4_{0}5_{0}\infty_{1}&2_{0}6_{2}\infty_{2}&1_{0}2_{1}5_{1}&2_{2}3_{1}4_{1}&3_{0}5_{2}6_{0}&1_{1}4_{2}6_{1}\\

1_{1}3_{0}\infty_{0}&4_{1}5_{2}\infty_{1}&2_{1}6_{1}\infty_{2}&1_{0}2_{2}5_{0}&2_{0}3_{1}4_{2}&3_{2}5_{1}6_{2}&1_{2}4_{0}6_{0}\\

1_{1}3_{1}\infty_{1}&4_{2}5_{2}\infty_{0}&2_{2}6_{0}\infty_{2}&1_{2}2_{0}5_{1}&2_{1}3_{2}4_{1}&3_{0}5_{0}6_{2}&1_{0}4_{0}6_{1}\\

1_{1}3_{2}\infty_{2}&4_{0}5_{1}\infty_{0}&2_{2}6_{1}\infty_{1}&1_{2}2_{1}5_{0}&2_{0}3_{0}4_{1}&3_{1}5_{2}6_{2}&1_{0}4_{2}6_{0}\\

1_{0}3_{2}\infty_{1}&4_{0}5_{2}\infty_{2}&2_{2}6_{2}\infty_{0}&1_{1}2_{0}5_{0}&2_{1}3_{0}4_{2}&3_{1}5_{1}6_{0}&1_{2}4_{1}6_{1}\\

1_{0}3_{1}\infty_{0}&4_{2}5_{0}\infty_{2}&2_{1}6_{2}\infty_{1}&1_{2}2_{2}5_{2}&2_{0}3_{2}4_{0}&3_{0}5_{1}6_{1}&1_{1}4_{1}6_{0}\\

1_{2}3_{1}\infty_{2}&4_{2}5_{1}\infty_{1}&2_{0}6_{1}\infty_{0}&1_{1}2_{1}5_{2}&2_{2}3_{0}4_{0}&3_{2}5_{0}6_{0}&1_{0}4_{1}6_{2}\\

1_{0}3_{0}\infty_{2}&4_{1}5_{0}\infty_{0}&2_{0}6_{0}\infty_{1}&1_{1}2_{2}5_{1}&2_{1}3_{1}4_{0}&3_{2}5_{2}6_{1}&1_{2}4_{2}6_{2}
\end{array}$$

point $1_{0}$:
$$\begin{array}{llllllll}
5_{1}6_{0}\infty_{2}&0_{0}2_{0}5_{0}&2_{1}3_{0}6_{2}&3_{1}4_{0}5_{2}&0_{1}4_{2}6_{1}&0_{2}3_{2}\infty_{1}&2_{2}4_{1}\infty_{0}\\

5_{0}6_{0}\infty_{0}&0_{0}2_{1}5_{2}&2_{0}3_{2}6_{1}&3_{1}4_{1}5_{1}&0_{1}4_{0}6_{2}&0_{2}3_{0}\infty_{2}&2_{2}4_{2}\infty_{1}\\

5_{2}6_{2}\infty_{2}&0_{0}2_{2}5_{1}&2_{1}3_{1}6_{1}&3_{0}4_{0}5_{0}&0_{2}4_{2}6_{0}&0_{1}3_{2}\infty_{0}&2_{0}4_{1}\infty_{1}\\

5_{1}6_{1}\infty_{1}&0_{1}2_{1}5_{0}&2_{2}3_{2}6_{2}&3_{0}4_{1}5_{2}&0_{0}4_{0}6_{0}&0_{2}3_{1}\infty_{0}&2_{0}4_{2}\infty_{2}\\

5_{1}6_{2}\infty_{0}&0_{2}2_{0}5_{2}&2_{1}3_{2}6_{0}&3_{1}4_{2}5_{0}&0_{0}4_{1}6_{1}&0_{1}3_{0}\infty_{1}&2_{2}4_{0}\infty_{2}\\

5_{2}6_{0}\infty_{1}&0_{2}2_{1}5_{1}&2_{2}3_{0}6_{1}&3_{2}4_{1}5_{0}&0_{0}4_{2}6_{2}&0_{1}3_{1}\infty_{2}&2_{0}4_{0}\infty_{0}\\

5_{0}6_{2}\infty_{1}&0_{1}2_{0}5_{1}&2_{2}3_{1}6_{0}&3_{2}4_{2}5_{2}&0_{2}4_{0}6_{1}&0_{0}3_{0}\infty_{0}&2_{1}4_{1}\infty_{2}\\

5_{2}6_{1}\infty_{0}&0_{2}2_{2}5_{0}&2_{0}3_{1}6_{2}&3_{0}4_{2}5_{1}&0_{1}4_{1}6_{0}&0_{0}3_{2}\infty_{2}&2_{1}4_{0}\infty_{1}\\

5_{0}6_{1}\infty_{2}&0_{1}2_{2}5_{2}&2_{0}3_{0}6_{0}&3_{2}4_{0}5_{1}&0_{2}4_{1}6_{2}&0_{0}3_{1}\infty_{1}&2_{1}4_{2}\infty_{0}
\end{array}$$

point $1_{1}$:
$$\begin{array}{llllllll}
5_{0}6_{2}\infty_{0}&0_{0}2_{0}5_{1}&2_{1}3_{0}6_{1}&3_{1}4_{2}5_{2}&0_{1}4_{0}6_{0}&0_{2}3_{2}\infty_{2}&2_{2}4_{1}\infty_{1}\\

5_{1}6_{2}\infty_{2}&0_{0}2_{1}5_{0}&2_{2}3_{0}6_{0}&3_{2}4_{1}5_{2}&0_{2}4_{2}6_{1}&0_{1}3_{1}\infty_{0}&2_{0}4_{0}\infty_{1}\\

5_{1}6_{1}\infty_{0}&0_{0}2_{2}5_{2}&2_{1}3_{1}6_{0}&3_{0}4_{2}5_{0}&0_{2}4_{0}6_{2}&0_{1}3_{2}\infty_{1}&2_{0}4_{1}\infty_{2}\\

5_{2}6_{2}\infty_{1}&0_{2}2_{2}5_{1}&2_{0}3_{1}6_{1}&3_{2}4_{0}5_{0}&0_{0}4_{2}6_{0}&0_{1}3_{0}\infty_{2}&2_{1}4_{1}\infty_{0}\\

5_{1}6_{0}\infty_{1}&0_{1}2_{0}5_{2}&2_{1}3_{2}6_{2}&3_{1}4_{1}5_{0}&0_{0}4_{0}6_{1}&0_{2}3_{0}\infty_{0}&2_{2}4_{2}\infty_{2}\\

5_{0}6_{0}\infty_{2}&0_{1}2_{1}5_{1}&2_{2}3_{2}6_{1}&3_{0}4_{0}5_{2}&0_{0}4_{1}6_{2}&0_{2}3_{1}\infty_{1}&2_{0}4_{2}\infty_{0}\\

5_{2}6_{1}\infty_{2}&0_{1}2_{2}5_{0}&2_{0}3_{0}6_{2}&3_{1}4_{0}5_{1}&0_{2}4_{1}6_{0}&0_{0}3_{2}\infty_{0}&2_{1}4_{2}\infty_{1}\\

5_{0}6_{1}\infty_{1}&0_{2}2_{1}5_{2}&2_{0}3_{2}6_{0}&3_{0}4_{1}5_{1}&0_{1}4_{2}6_{2}&0_{0}3_{1}\infty_{2}&2_{2}4_{0}\infty_{0}\\

5_{2}6_{0}\infty_{0}&0_{2}2_{0}5_{0}&2_{2}3_{1}6_{2}&3_{2}4_{2}5_{1}&0_{1}4_{1}6_{1}&0_{0}3_{0}\infty_{1}&2_{1}4_{0}\infty_{2}
\end{array}$$

point $1_{2}$:
$$\begin{array}{llllllll}
5_{0}6_{1}\infty_{0}&0_{0}2_{0}5_{2}&2_{2}3_{0}6_{2}&3_{2}4_{1}5_{1}&0_{2}4_{0}6_{0}&0_{1}3_{1}\infty_{1}&2_{1}4_{2}\infty_{2}\\

5_{2}6_{2}\infty_{0}&0_{0}2_{1}5_{1}&2_{2}3_{2}6_{0}&3_{0}4_{1}5_{0}&0_{1}4_{0}6_{1}&0_{2}3_{1}\infty_{2}&2_{0}4_{2}\infty_{1}\\

5_{2}6_{1}\infty_{1}&0_{0}2_{2}5_{0}&2_{1}3_{0}6_{0}&3_{1}4_{2}5_{1}&0_{1}4_{1}6_{2}&0_{2}3_{2}\infty_{0}&2_{0}4_{0}\infty_{2}\\

5_{1}6_{1}\infty_{2}&0_{1}2_{1}5_{2}&2_{0}3_{2}6_{2}&3_{1}4_{0}5_{0}&0_{0}4_{1}6_{0}&0_{2}3_{0}\infty_{1}&2_{2}4_{2}\infty_{0}\\

5_{0}6_{0}\infty_{1}&0_{2}2_{2}5_{2}&2_{1}3_{1}6_{2}&3_{0}4_{0}5_{1}&0_{0}4_{2}6_{1}&0_{1}3_{2}\infty_{2}&2_{0}4_{1}\infty_{0}\\

5_{2}6_{0}\infty_{2}&0_{2}2_{0}5_{1}&2_{2}3_{1}6_{1}&3_{2}4_{2}5_{0}&0_{0}4_{0}6_{2}&0_{1}3_{0}\infty_{0}&2_{1}4_{1}\infty_{1}\\

5_{1}6_{2}\infty_{1}&0_{2}2_{1}5_{0}&2_{0}3_{0}6_{1}&3_{2}4_{0}5_{2}&0_{1}4_{2}6_{0}&0_{0}3_{1}\infty_{0}&2_{2}4_{1}\infty_{2}\\

5_{1}6_{0}\infty_{0}&0_{1}2_{0}5_{0}&2_{1}3_{2}6_{1}&3_{1}4_{1}5_{2}&0_{2}4_{2}6_{2}&0_{0}3_{0}\infty_{2}&2_{2}4_{0}\infty_{1}\\

5_{0}6_{2}\infty_{2}&0_{1}2_{2}5_{1}&2_{0}3_{1}6_{0}&3_{0}4_{2}5_{2}&0_{2}4_{1}6_{1}&0_{0}3_{2}\infty_{1}&2_{1}4_{0}\infty_{0}
\end{array}$$

point $2_{0}$:
$$\begin{array}{llllllll}
3_{2}5_{0}\infty_{1}&0_{0}6_{0}\infty_{0}&0_{1}1_{0}5_{1}&1_{2}3_{0}6_{1}&0_{2}3_{1}4_{2}&1_{1}4_{1}\infty_{2}&4_{0}5_{2}6_{2}\\

3_{1}5_{0}\infty_{2}&0_{0}6_{2}\infty_{1}&0_{2}1_{0}5_{2}&1_{1}3_{2}6_{0}&0_{1}3_{0}4_{2}&1_{2}4_{1}\infty_{0}&4_{0}5_{1}6_{1}\\

3_{0}5_{0}\infty_{0}&0_{0}6_{1}\infty_{2}&0_{1}1_{1}5_{2}&1_{2}3_{1}6_{0}&0_{2}3_{2}4_{0}&1_{0}4_{1}\infty_{1}&4_{2}5_{1}6_{2}\\

3_{1}5_{1}\infty_{1}&0_{2}6_{1}\infty_{0}&0_{0}1_{0}5_{0}&1_{1}3_{0}6_{2}&0_{1}3_{2}4_{1}&1_{2}4_{0}\infty_{2}&4_{2}5_{2}6_{0}\\

3_{1}5_{2}\infty_{0}&0_{1}6_{1}\infty_{1}&0_{0}1_{1}5_{1}&1_{2}3_{2}6_{2}&0_{2}3_{0}4_{1}&1_{0}4_{2}\infty_{2}&4_{0}5_{0}6_{0}\\

3_{0}5_{1}\infty_{2}&0_{2}6_{0}\infty_{1}&0_{0}1_{2}5_{2}&1_{0}3_{2}6_{1}&0_{1}3_{1}4_{0}&1_{1}4_{2}\infty_{0}&4_{1}5_{0}6_{2}\\

3_{2}5_{1}\infty_{0}&0_{1}6_{0}\infty_{2}&0_{2}1_{1}5_{0}&1_{0}3_{1}6_{2}&0_{0}3_{0}4_{0}&1_{2}4_{2}\infty_{1}&4_{1}5_{2}6_{1}\\

3_{2}5_{2}\infty_{2}&0_{1}6_{2}\infty_{0}&0_{2}1_{2}5_{1}&1_{0}3_{0}6_{0}&0_{0}3_{1}4_{1}&1_{1}4_{0}\infty_{1}&4_{2}5_{0}6_{1}\\

3_{0}5_{2}\infty_{1}&0_{2}6_{2}\infty_{2}&0_{1}1_{2}5_{0}&1_{1}3_{1}6_{1}&0_{0}3_{2}4_{2}&1_{0}4_{0}\infty_{0}&4_{1}5_{1}6_{0}
\end{array}$$

point $2_{1}$:
$$\begin{array}{llllllll}
3_{1}5_{2}\infty_{2}&0_{0}6_{2}\infty_{0}&0_{2}1_{0}5_{1}&1_{2}3_{2}6_{1}&0_{1}3_{0}4_{0}&1_{1}4_{2}\infty_{1}&4_{1}5_{0}6_{0}\\

3_{2}5_{0}\infty_{0}&0_{0}6_{1}\infty_{1}&0_{1}1_{1}5_{1}&1_{2}3_{1}6_{2}&0_{2}3_{0}4_{2}&1_{0}4_{1}\infty_{2}&4_{0}5_{2}6_{0}\\

3_{1}5_{1}\infty_{0}&0_{0}6_{0}\infty_{2}&0_{2}1_{2}5_{0}&1_{1}3_{0}6_{1}&0_{1}3_{2}4_{2}&1_{0}4_{0}\infty_{1}&4_{1}5_{2}6_{2}\\

3_{2}5_{2}\infty_{1}&0_{1}6_{2}\infty_{2}&0_{0}1_{1}5_{0}&1_{2}3_{0}6_{0}&0_{2}3_{1}4_{0}&1_{0}4_{2}\infty_{0}&4_{1}5_{1}6_{1}\\

3_{2}5_{1}\infty_{2}&0_{2}6_{2}\infty_{1}&0_{1}1_{0}5_{0}&1_{1}3_{1}6_{0}&0_{0}3_{0}4_{1}&1_{2}4_{0}\infty_{0}&4_{2}5_{2}6_{1}\\

3_{1}5_{0}\infty_{1}&0_{2}6_{1}\infty_{2}&0_{1}1_{2}5_{2}&1_{0}3_{0}6_{2}&0_{0}3_{2}4_{0}&1_{1}4_{1}\infty_{0}&4_{2}5_{1}6_{0}\\

3_{0}5_{0}\infty_{2}&0_{1}6_{1}\infty_{0}&0_{2}1_{1}5_{2}&1_{0}3_{2}6_{0}&0_{0}3_{1}4_{2}&1_{2}4_{1}\infty_{1}&4_{0}5_{1}6_{2}\\

3_{0}5_{2}\infty_{0}&0_{1}6_{0}\infty_{1}&0_{0}1_{2}5_{1}&1_{0}3_{1}6_{1}&0_{2}3_{2}4_{1}&1_{1}4_{0}\infty_{2}&4_{2}5_{0}6_{2}\\

3_{0}5_{1}\infty_{1}&0_{2}6_{0}\infty_{0}&0_{0}1_{0}5_{2}&1_{1}3_{2}6_{2}&0_{1}3_{1}4_{1}&1_{2}4_{2}\infty_{2}&4_{0}5_{0}6_{1}
\end{array}$$

point $2_{2}$:
$$\begin{array}{llllllll}
3_{2}5_{1}\infty_{1}&0_{0}6_{1}\infty_{0}&0_{1}1_{0}5_{2}&1_{1}3_{1}6_{2}&0_{2}3_{0}4_{0}&1_{2}4_{1}\infty_{2}&4_{2}5_{0}6_{0}\\

3_{2}5_{2}\infty_{0}&0_{0}6_{0}\infty_{1}&0_{1}1_{1}5_{0}&1_{2}3_{0}6_{2}&0_{2}3_{1}4_{1}&1_{0}4_{0}\infty_{2}&4_{2}5_{1}6_{1}\\

3_{1}5_{0}\infty_{0}&0_{0}6_{2}\infty_{2}&0_{2}1_{2}5_{2}&1_{1}3_{2}6_{1}&0_{1}3_{0}4_{1}&1_{0}4_{2}\infty_{1}&4_{0}5_{1}6_{0}\\

3_{0}5_{1}\infty_{0}&0_{2}6_{0}\infty_{2}&0_{0}1_{1}5_{2}&1_{0}3_{2}6_{2}&0_{1}3_{1}4_{2}&1_{2}4_{0}\infty_{1}&4_{1}5_{0}6_{1}\\

3_{0}5_{2}\infty_{2}&0_{1}6_{0}\infty_{0}&0_{0}1_{0}5_{1}&1_{2}3_{1}6_{1}&0_{2}3_{2}4_{2}&1_{1}4_{1}\infty_{1}&4_{0}5_{0}6_{2}\\

3_{2}5_{0}\infty_{2}&0_{1}6_{2}\infty_{1}&0_{2}1_{1}5_{1}&1_{0}3_{0}6_{1}&0_{0}3_{1}4_{0}&1_{2}4_{2}\infty_{0}&4_{1}5_{2}6_{0}\\

3_{1}5_{2}\infty_{1}&0_{1}6_{1}\infty_{2}&0_{2}1_{0}5_{0}&1_{2}3_{2}6_{0}&0_{0}3_{0}4_{2}&1_{1}4_{0}\infty_{0}&4_{1}5_{1}6_{2}\\

3_{1}5_{1}\infty_{2}&0_{2}6_{1}\infty_{1}&0_{0}1_{2}5_{0}&1_{1}3_{0}6_{0}&0_{1}3_{2}4_{0}&1_{0}4_{1}\infty_{0}&4_{2}5_{2}6_{2}\\

3_{0}5_{0}\infty_{1}&0_{2}6_{2}\infty_{0}&0_{1}1_{2}5_{1}&1_{0}3_{1}6_{0}&0_{0}3_{2}4_{1}&1_{1}4_{2}\infty_{2}&4_{0}5_{2}6_{1}
\end{array}$$

point $3_{0}$:
$$\begin{array}{llllllll}
2_{2}5_{0}\infty_{1}&4_{2}6_{1}\infty_{0}&1_{2}2_{1}6_{0}&0_{0}2_{0}4_{0}&1_{1}4_{1}5_{1}&0_{1}5_{2}6_{2}&0_{2}1_{0}\infty_{2}\\

2_{0}5_{1}\infty_{2}&4_{0}6_{2}\infty_{1}&1_{0}2_{2}6_{1}&0_{0}2_{1}4_{1}&1_{1}4_{2}5_{0}&0_{2}5_{2}6_{0}&0_{1}1_{2}\infty_{0}\\

2_{1}5_{2}\infty_{0}&4_{1}6_{1}\infty_{1}&1_{0}2_{0}6_{0}&0_{0}2_{2}4_{2}&1_{2}4_{0}5_{1}&0_{2}5_{0}6_{2}&0_{1}1_{1}\infty_{2}\\

2_{0}5_{0}\infty_{0}&4_{1}6_{0}\infty_{2}&1_{1}2_{1}6_{1}&0_{2}2_{2}4_{0}&1_{2}4_{2}5_{2}&0_{0}5_{1}6_{2}&0_{1}1_{0}\infty_{1}\\

2_{1}5_{1}\infty_{1}&4_{0}6_{1}\infty_{2}&1_{2}2_{2}6_{2}&0_{1}2_{0}4_{2}&1_{0}4_{1}5_{2}&0_{0}5_{0}6_{0}&0_{2}1_{1}\infty_{0}\\

2_{1}5_{0}\infty_{2}&4_{0}6_{0}\infty_{0}&1_{1}2_{0}6_{2}&0_{1}2_{2}4_{1}&1_{0}4_{2}5_{1}&0_{0}5_{2}6_{1}&0_{2}1_{2}\infty_{1}\\

2_{2}5_{1}\infty_{0}&4_{2}6_{0}\infty_{1}&1_{0}2_{1}6_{2}&0_{2}2_{0}4_{1}&1_{1}4_{0}5_{2}&0_{1}5_{0}6_{1}&0_{0}1_{2}\infty_{2}\\

2_{2}5_{2}\infty_{2}&4_{1}6_{2}\infty_{0}&1_{2}2_{0}6_{1}&0_{2}2_{1}4_{2}&1_{0}4_{0}5_{0}&0_{1}5_{1}6_{0}&0_{0}1_{1}\infty_{1}\\

2_{0}5_{2}\infty_{1}&4_{2}6_{2}\infty_{2}&1_{1}2_{2}6_{0}&0_{1}2_{1}4_{0}&1_{2}4_{1}5_{0}&0_{2}5_{1}6_{1}&0_{0}1_{0}\infty_{0}
\end{array}$$

point $3_{1}$:
$$\begin{array}{llllllll}
2_{2}5_{0}\infty_{0}&4_{0}6_{1}\infty_{1}&1_{2}2_{1}6_{2}&0_{0}2_{0}4_{1}&1_{1}4_{2}5_{2}&0_{2}5_{1}6_{0}&0_{1}1_{0}\infty_{2}\\

2_{0}5_{1}\infty_{1}&4_{1}6_{1}\infty_{0}&1_{1}2_{2}6_{2}&0_{0}2_{1}4_{2}&1_{0}4_{0}5_{2}&0_{1}5_{0}6_{0}&0_{2}1_{2}\infty_{2}\\

2_{1}5_{1}\infty_{0}&4_{1}6_{2}\infty_{2}&1_{2}2_{0}6_{0}&0_{0}2_{2}4_{0}&1_{0}4_{2}5_{0}&0_{1}5_{2}6_{1}&0_{2}1_{1}\infty_{1}\\

2_{0}5_{0}\infty_{2}&4_{2}6_{2}\infty_{1}&1_{0}2_{2}6_{0}&0_{2}2_{1}4_{0}&1_{2}4_{1}5_{2}&0_{0}5_{1}6_{1}&0_{1}1_{1}\infty_{0}\\

2_{2}5_{2}\infty_{1}&4_{0}6_{0}\infty_{2}&1_{1}2_{0}6_{1}&0_{1}2_{1}4_{1}&1_{2}4_{2}5_{1}&0_{0}5_{0}6_{2}&0_{2}1_{0}\infty_{0}\\

2_{1}5_{0}\infty_{1}&4_{2}6_{0}\infty_{0}&1_{2}2_{2}6_{1}&0_{1}2_{0}4_{0}&1_{0}4_{1}5_{1}&0_{2}5_{2}6_{2}&0_{0}1_{1}\infty_{2}\\

2_{0}5_{2}\infty_{0}&4_{2}6_{1}\infty_{2}&1_{1}2_{1}6_{0}&0_{2}2_{2}4_{1}&1_{2}4_{0}5_{0}&0_{1}5_{1}6_{2}&0_{0}1_{0}\infty_{1}\\

2_{2}5_{1}\infty_{2}&4_{0}6_{2}\infty_{0}&1_{0}2_{1}6_{1}&0_{2}2_{0}4_{2}&1_{1}4_{1}5_{0}&0_{0}5_{2}6_{0}&0_{1}1_{2}\infty_{1}\\

2_{1}5_{2}\infty_{2}&4_{1}6_{0}\infty_{1}&1_{0}2_{0}6_{2}&0_{1}2_{2}4_{2}&1_{1}4_{0}5_{1}&0_{2}5_{0}6_{1}&0_{0}1_{2}\infty_{0}
\end{array}$$

point $3_{2}$:
$$\begin{array}{llllllll}
2_{1}5_{0}\infty_{0}&4_{0}6_{2}\infty_{2}&1_{1}2_{2}6_{1}&0_{0}2_{0}4_{2}&1_{2}4_{1}5_{1}&0_{1}5_{2}6_{0}&0_{2}1_{0}\infty_{1}\\

2_{2}5_{2}\infty_{0}&4_{2}6_{1}\infty_{1}&1_{1}2_{0}6_{0}&0_{0}2_{1}4_{0}&1_{0}4_{1}5_{0}&0_{2}5_{1}6_{2}&0_{1}1_{2}\infty_{2}\\

2_{0}5_{2}\infty_{2}&4_{2}6_{2}\infty_{0}&1_{2}2_{1}6_{1}&0_{0}2_{2}4_{1}&1_{0}4_{0}5_{1}&0_{2}5_{0}6_{0}&0_{1}1_{1}\infty_{1}\\

2_{1}5_{2}\infty_{1}&4_{0}6_{1}\infty_{0}&1_{0}2_{2}6_{2}&0_{1}2_{0}4_{1}&1_{2}4_{2}5_{0}&0_{0}5_{1}6_{0}&0_{2}1_{1}\infty_{2}\\

2_{0}5_{0}\infty_{1}&4_{1}6_{1}\infty_{2}&1_{0}2_{1}6_{0}&0_{1}2_{2}4_{0}&1_{1}4_{2}5_{1}&0_{0}5_{2}6_{2}&0_{2}1_{2}\infty_{0}\\

2_{0}5_{1}\infty_{0}&4_{1}6_{2}\infty_{1}&1_{2}2_{2}6_{0}&0_{1}2_{1}4_{2}&1_{1}4_{0}5_{0}&0_{2}5_{2}6_{1}&0_{0}1_{0}\infty_{2}\\

2_{2}5_{0}\infty_{2}&4_{1}6_{0}\infty_{0}&1_{1}2_{1}6_{2}&0_{2}2_{0}4_{0}&1_{0}4_{2}5_{2}&0_{1}5_{1}6_{1}&0_{0}1_{2}\infty_{1}\\

2_{1}5_{1}\infty_{2}&4_{0}6_{0}\infty_{1}&1_{2}2_{0}6_{2}&0_{2}2_{2}4_{2}&1_{1}4_{1}5_{2}&0_{0}5_{0}6_{1}&0_{1}1_{0}\infty_{0}\\

2_{2}5_{1}\infty_{1}&4_{2}6_{0}\infty_{2}&1_{0}2_{0}6_{1}&0_{2}2_{1}4_{1}&1_{2}4_{0}5_{2}&0_{1}5_{0}6_{2}&0_{0}1_{1}\infty_{0}
\end{array}$$

point $4_{0}$:
$$\begin{array}{llllllll}
3_{0}6_{2}\infty_{1}&0_{0}5_{0}\infty_{0}&0_{1}2_{2}3_{2}&1_{0}3_{1}5_{2}&0_{2}1_{2}6_{0}&1_{1}2_{1}\infty_{2}&2_{0}5_{1}6_{1}\\

3_{1}6_{2}\infty_{0}&0_{0}5_{2}\infty_{1}&0_{2}2_{2}3_{0}&1_{0}3_{2}5_{1}&0_{1}1_{1}6_{0}&1_{2}2_{0}\infty_{2}&2_{1}5_{0}6_{1}\\

3_{2}6_{0}\infty_{1}&0_{0}5_{1}\infty_{2}&0_{1}2_{1}3_{0}&1_{2}3_{1}5_{0}&0_{2}1_{1}6_{2}&1_{0}2_{0}\infty_{0}&2_{2}5_{2}6_{1}\\

3_{0}6_{0}\infty_{0}&0_{1}5_{0}\infty_{2}&0_{0}2_{1}3_{2}&1_{1}3_{1}5_{1}&0_{2}1_{0}6_{1}&1_{2}2_{2}\infty_{1}&2_{0}5_{2}6_{2}\\

3_{1}6_{1}\infty_{1}&0_{2}5_{2}\infty_{2}&0_{0}2_{0}3_{0}&1_{1}3_{2}5_{0}&0_{1}1_{0}6_{2}&1_{2}2_{1}\infty_{0}&2_{2}5_{1}6_{0}\\

3_{1}6_{0}\infty_{2}&0_{1}5_{2}\infty_{0}&0_{2}2_{0}3_{2}&1_{2}3_{0}5_{1}&0_{0}1_{1}6_{1}&1_{0}2_{1}\infty_{1}&2_{2}5_{0}6_{2}\\

3_{0}6_{1}\infty_{2}&0_{2}5_{0}\infty_{1}&0_{1}2_{0}3_{1}&1_{2}3_{2}5_{2}&0_{0}1_{0}6_{0}&1_{1}2_{2}\infty_{0}&2_{1}5_{1}6_{2}\\

3_{2}6_{2}\infty_{2}&0_{2}5_{1}\infty_{0}&0_{0}2_{2}3_{1}&1_{0}3_{0}5_{0}&0_{1}1_{2}6_{1}&1_{1}2_{0}\infty_{1}&2_{1}5_{2}6_{0}\\

3_{2}6_{1}\infty_{0}&0_{1}5_{1}\infty_{1}&0_{2}2_{1}3_{1}&1_{1}3_{0}5_{2}&0_{0}1_{2}6_{2}&1_{0}2_{2}\infty_{2}&2_{0}5_{0}6_{0}
\end{array}$$

point $4_{1}$:
$$\begin{array}{llllllll}
3_{2}6_{1}\infty_{2}&0_{0}5_{2}\infty_{0}&0_{1}2_{2}3_{0}&1_{0}3_{1}5_{1}&0_{2}1_{1}6_{0}&1_{2}2_{1}\infty_{1}&2_{0}5_{0}6_{2}\\

3_{0}6_{0}\infty_{2}&0_{0}5_{1}\infty_{1}&0_{2}2_{2}3_{1}&1_{0}3_{2}5_{0}&0_{1}1_{2}6_{2}&1_{1}2_{1}\infty_{0}&2_{0}5_{2}6_{1}\\

3_{1}6_{1}\infty_{0}&0_{0}5_{0}\infty_{2}&0_{2}2_{0}3_{0}&1_{2}3_{2}5_{1}&0_{1}1_{0}6_{0}&1_{1}2_{2}\infty_{1}&2_{1}5_{2}6_{2}\\

3_{2}6_{2}\infty_{1}&0_{1}5_{2}\infty_{2}&0_{0}2_{1}3_{0}&1_{1}3_{1}5_{0}&0_{2}1_{2}6_{1}&1_{0}2_{2}\infty_{0}&2_{0}5_{1}6_{0}\\

3_{0}6_{2}\infty_{0}&0_{2}5_{1}\infty_{2}&0_{1}2_{1}3_{1}&1_{1}3_{2}5_{2}&0_{0}1_{2}6_{0}&1_{0}2_{0}\infty_{1}&2_{2}5_{0}6_{1}\\

3_{1}6_{2}\infty_{2}&0_{1}5_{0}\infty_{1}&0_{2}2_{1}3_{2}&1_{1}3_{0}5_{1}&0_{0}1_{0}6_{1}&1_{2}2_{0}\infty_{0}&2_{2}5_{2}6_{0}\\

3_{0}6_{1}\infty_{1}&0_{1}5_{1}\infty_{0}&0_{0}2_{2}3_{2}&1_{2}3_{1}5_{2}&0_{2}1_{0}6_{2}&1_{1}2_{0}\infty_{2}&2_{1}5_{0}6_{0}\\

3_{1}6_{0}\infty_{1}&0_{2}5_{0}\infty_{0}&0_{1}2_{0}3_{2}&1_{0}3_{0}5_{2}&0_{0}1_{1}6_{2}&1_{2}2_{2}\infty_{2}&2_{1}5_{1}6_{1}\\

3_{2}6_{0}\infty_{0}&0_{2}5_{2}\infty_{1}&0_{0}2_{0}3_{1}&1_{2}3_{0}5_{0}&0_{1}1_{1}6_{1}&1_{0}2_{1}\infty_{2}&2_{2}5_{1}6_{2}
\end{array}$$

point $4_{2}$:
$$\begin{array}{llllllll}
3_{1}6_{2}\infty_{1}&0_{0}5_{1}\infty_{0}&0_{2}2_{1}3_{0}&1_{0}3_{2}5_{2}&0_{1}1_{2}6_{0}&1_{1}2_{2}\infty_{2}&2_{0}5_{0}6_{1}\\

3_{2}6_{2}\infty_{0}&0_{0}5_{0}\infty_{1}&0_{1}2_{0}3_{0}&1_{1}3_{1}5_{2}&0_{2}1_{0}6_{0}&1_{2}2_{1}\infty_{2}&2_{2}5_{1}6_{1}\\

3_{0}6_{0}\infty_{1}&0_{0}5_{2}\infty_{2}&0_{1}2_{1}3_{2}&1_{0}3_{1}5_{0}&0_{2}1_{1}6_{1}&1_{2}2_{2}\infty_{0}&2_{0}5_{1}6_{2}\\

3_{0}6_{2}\infty_{2}&0_{2}5_{2}\infty_{0}&0_{0}2_{1}3_{1}&1_{1}3_{2}5_{1}&0_{1}1_{0}6_{1}&1_{2}2_{0}\infty_{1}&2_{2}5_{0}6_{0}\\

3_{1}6_{0}\infty_{0}&0_{1}5_{1}\infty_{2}&0_{0}2_{0}3_{2}&1_{1}3_{0}5_{0}&0_{2}1_{2}6_{2}&1_{0}2_{2}\infty_{1}&2_{1}5_{2}6_{1}\\

3_{0}6_{1}\infty_{0}&0_{1}5_{2}\infty_{1}&0_{2}2_{2}3_{2}&1_{2}3_{1}5_{1}&0_{0}1_{1}6_{0}&1_{0}2_{0}\infty_{2}&2_{1}5_{0}6_{2}\\

3_{1}6_{1}\infty_{2}&0_{2}5_{1}\infty_{1}&0_{0}2_{2}3_{0}&1_{2}3_{2}5_{0}&0_{1}1_{1}6_{2}&1_{0}2_{1}\infty_{0}&2_{0}5_{2}6_{0}\\

3_{2}6_{1}\infty_{1}&0_{2}5_{0}\infty_{2}&0_{1}2_{2}3_{1}&1_{2}3_{0}5_{2}&0_{0}1_{0}6_{2}&1_{1}2_{0}\infty_{0}&2_{1}5_{1}6_{0}\\

3_{2}6_{0}\infty_{2}&0_{1}5_{0}\infty_{0}&0_{2}2_{0}3_{1}&1_{0}3_{0}5_{1}&0_{0}1_{2}6_{1}&1_{1}2_{1}\infty_{1}&2_{2}5_{2}6_{2}
\end{array}$$

point $5_{0}$:
$$\begin{array}{llllllll}
2_{2}3_{2}\infty_{2}&0_{0}4_{0}\infty_{0}&1_{0}6_{2}\infty_{1}&0_{1}1_{2}2_{0}&1_{1}3_{0}4_{2}&0_{2}3_{1}6_{1}&2_{1}4_{1}6_{0}\\

2_{1}3_{0}\infty_{2}&0_{0}4_{2}\infty_{1}&1_{2}6_{1}\infty_{0}&0_{2}1_{0}2_{2}&1_{1}3_{2}4_{0}&0_{1}3_{1}6_{0}&2_{0}4_{1}6_{2}\\

2_{1}3_{2}\infty_{0}&0_{0}4_{1}\infty_{2}&1_{2}6_{0}\infty_{1}&0_{2}1_{1}2_{0}&1_{0}3_{1}4_{2}&0_{1}3_{0}6_{1}&2_{2}4_{0}6_{2}\\

2_{2}3_{1}\infty_{0}&0_{1}4_{1}\infty_{1}&1_{0}6_{1}\infty_{2}&0_{0}1_{1}2_{1}&1_{2}3_{2}4_{2}&0_{2}3_{0}6_{2}&2_{0}4_{0}6_{0}\\

2_{2}3_{0}\infty_{1}&0_{1}4_{2}\infty_{0}&1_{2}6_{2}\infty_{2}&0_{0}1_{0}2_{0}&1_{1}3_{1}4_{1}&0_{2}3_{2}6_{0}&2_{1}4_{0}6_{1}\\

2_{1}3_{1}\infty_{1}&0_{2}4_{1}\infty_{0}&1_{1}6_{0}\infty_{2}&0_{0}1_{2}2_{2}&1_{0}3_{0}4_{0}&0_{1}3_{2}6_{2}&2_{0}4_{2}6_{1}\\

2_{0}3_{2}\infty_{1}&0_{2}4_{2}\infty_{2}&1_{1}6_{2}\infty_{0}&0_{1}1_{0}2_{1}&1_{2}3_{1}4_{0}&0_{0}3_{0}6_{0}&2_{2}4_{1}6_{1}\\

2_{0}3_{0}\infty_{0}&0_{1}4_{0}\infty_{2}&1_{1}6_{1}\infty_{1}&0_{2}1_{2}2_{1}&1_{0}3_{2}4_{1}&0_{0}3_{1}6_{2}&2_{2}4_{2}6_{0}\\

2_{0}3_{1}\infty_{2}&0_{2}4_{0}\infty_{1}&1_{0}6_{0}\infty_{0}&0_{1}1_{1}2_{2}&1_{2}3_{0}4_{1}&0_{0}3_{2}6_{1}&2_{1}4_{2}6_{2}
\end{array}$$

point $5_{1}$:
$$\begin{array}{llllllll}
2_{2}3_{2}\infty_{1}&0_{0}4_{2}\infty_{0}&1_{2}6_{1}\infty_{2}&0_{1}1_{0}2_{0}&1_{1}3_{0}4_{1}&0_{2}3_{1}6_{0}&2_{1}4_{0}6_{2}\\

2_{1}3_{1}\infty_{0}&0_{0}4_{1}\infty_{1}&1_{1}6_{2}\infty_{2}&0_{2}1_{2}2_{0}&1_{0}3_{0}4_{2}&0_{1}3_{2}6_{1}&2_{2}4_{0}6_{0}\\

2_{2}3_{0}\infty_{0}&0_{0}4_{0}\infty_{2}&1_{0}6_{1}\infty_{1}&0_{1}1_{1}2_{1}&1_{2}3_{1}4_{2}&0_{2}3_{2}6_{2}&2_{0}4_{1}6_{0}\\

2_{1}3_{2}\infty_{2}&0_{2}4_{0}\infty_{0}&1_{2}6_{2}\infty_{1}&0_{0}1_{1}2_{0}&1_{0}3_{1}4_{1}&0_{1}3_{0}6_{0}&2_{2}4_{2}6_{1}\\

2_{2}3_{1}\infty_{2}&0_{1}4_{1}\infty_{0}&1_{1}6_{0}\infty_{1}&0_{0}1_{2}2_{1}&1_{0}3_{2}4_{0}&0_{2}3_{0}6_{1}&2_{0}4_{2}6_{2}\\

2_{0}3_{2}\infty_{0}&0_{2}4_{2}\infty_{1}&1_{0}6_{0}\infty_{2}&0_{1}1_{2}2_{2}&1_{1}3_{1}4_{0}&0_{0}3_{0}6_{2}&2_{1}4_{1}6_{1}\\

2_{1}3_{0}\infty_{1}&0_{2}4_{1}\infty_{2}&1_{2}6_{0}\infty_{0}&0_{0}1_{0}2_{2}&1_{1}3_{2}4_{2}&0_{1}3_{1}6_{2}&2_{0}4_{0}6_{1}\\

2_{0}3_{1}\infty_{1}&0_{1}4_{2}\infty_{2}&1_{1}6_{1}\infty_{0}&0_{2}1_{0}2_{1}&1_{2}3_{0}4_{0}&0_{0}3_{2}6_{0}&2_{2}4_{1}6_{2}\\

2_{0}3_{0}\infty_{2}&0_{1}4_{0}\infty_{1}&1_{0}6_{2}\infty_{0}&0_{2}1_{1}2_{2}&1_{2}3_{2}4_{1}&0_{0}3_{1}6_{1}&2_{1}4_{2}6_{0}
\end{array}$$

point $5_{2}$:
$$\begin{array}{llllllll}
2_{1}3_{1}\infty_{2}&0_{0}4_{1}\infty_{0}&1_{1}6_{2}\infty_{1}&0_{2}1_{0}2_{0}&1_{2}3_{0}4_{2}&0_{1}3_{2}6_{0}&2_{2}4_{0}6_{1}\\

2_{1}3_{0}\infty_{0}&0_{0}4_{0}\infty_{1}&1_{0}6_{2}\infty_{2}&0_{2}1_{2}2_{2}&1_{1}3_{2}4_{1}&0_{1}3_{1}6_{1}&2_{0}4_{2}6_{0}\\

2_{2}3_{2}\infty_{0}&0_{0}4_{2}\infty_{2}&1_{2}6_{1}\infty_{1}&0_{1}1_{1}2_{0}&1_{0}3_{0}4_{1}&0_{2}3_{1}6_{2}&2_{1}4_{0}6_{0}\\

2_{2}3_{0}\infty_{2}&0_{1}4_{2}\infty_{1}&1_{1}6_{0}\infty_{0}&0_{0}1_{0}2_{1}&1_{2}3_{1}4_{1}&0_{2}3_{2}6_{1}&2_{0}4_{0}6_{2}\\
\end{array}$$
$$\begin{array}{llllllll}

2_{0}3_{0}\infty_{1}&0_{2}4_{0}\infty_{2}&1_{0}6_{1}\infty_{0}&0_{1}1_{2}2_{1}&1_{1}3_{1}4_{2}&0_{0}3_{2}6_{2}&2_{2}4_{1}6_{0}\\

2_{1}3_{2}\infty_{1}&0_{2}4_{2}\infty_{0}&1_{2}6_{0}\infty_{2}&0_{0}1_{1}2_{2}&1_{0}3_{1}4_{0}&0_{1}3_{0}6_{2}&2_{0}4_{1}6_{1}\\

2_{0}3_{2}\infty_{2}&0_{2}4_{1}\infty_{1}&1_{2}6_{2}\infty_{0}&0_{1}1_{0}2_{2}&1_{1}3_{0}4_{0}&0_{0}3_{1}6_{0}&2_{1}4_{2}6_{1}\\

2_{0}3_{1}\infty_{0}&0_{1}4_{1}\infty_{2}&1_{0}6_{0}\infty_{1}&0_{2}1_{1}2_{1}&1_{2}3_{2}4_{0}&0_{0}3_{0}6_{1}&2_{2}4_{2}6_{2}\\

2_{2}3_{1}\infty_{1}&0_{1}4_{0}\infty_{0}&1_{1}6_{1}\infty_{2}&0_{0}1_{2}2_{0}&1_{0}3_{2}4_{2}&0_{2}3_{0}6_{0}&2_{1}4_{1}6_{2}
\end{array}$$

point $6_{0}$:
$$\begin{array}{llllllll}
3_{0}4_{2}\infty_{1}&1_{2}5_{2}\infty_{2}&0_{0}2_{0}\infty_{0}&1_{0}2_{1}3_{2}&0_{1}3_{1}5_{0}&0_{2}1_{1}4_{1}&2_{2}4_{0}5_{1}\\

3_{0}4_{1}\infty_{2}&1_{0}5_{0}\infty_{0}&0_{0}2_{2}\infty_{1}&1_{1}2_{0}3_{2}&0_{2}3_{1}5_{1}&0_{1}1_{2}4_{2}&2_{1}4_{0}5_{2}\\

3_{1}4_{1}\infty_{1}&1_{2}5_{1}\infty_{0}&0_{0}2_{1}\infty_{2}&1_{1}2_{2}3_{0}&0_{1}3_{2}5_{2}&0_{2}1_{0}4_{2}&2_{0}4_{0}5_{0}\\

3_{2}4_{2}\infty_{2}&1_{0}5_{2}\infty_{1}&0_{1}2_{2}\infty_{0}&1_{1}2_{1}3_{1}&0_{0}3_{0}5_{0}&0_{2}1_{2}4_{0}&2_{0}4_{1}5_{1}\\

3_{2}4_{0}\infty_{1}&1_{1}5_{0}\infty_{2}&0_{2}2_{1}\infty_{0}&1_{0}2_{2}3_{1}&0_{1}3_{0}5_{1}&0_{0}1_{2}4_{1}&2_{0}4_{2}5_{2}\\

3_{0}4_{0}\infty_{0}&1_{0}5_{1}\infty_{2}&0_{1}2_{1}\infty_{1}&1_{2}2_{0}3_{1}&0_{2}3_{2}5_{0}&0_{0}1_{1}4_{2}&2_{2}4_{1}5_{2}\\

3_{1}4_{0}\infty_{2}&1_{1}5_{2}\infty_{0}&0_{2}2_{0}\infty_{1}&1_{2}2_{1}3_{0}&0_{0}3_{2}5_{1}&0_{1}1_{0}4_{1}&2_{2}4_{2}5_{0}\\

3_{1}4_{2}\infty_{0}&1_{1}5_{1}\infty_{1}&0_{1}2_{0}\infty_{2}&1_{2}2_{2}3_{2}&0_{2}3_{0}5_{2}&0_{0}1_{0}4_{0}&2_{1}4_{1}5_{0}\\

3_{2}4_{1}\infty_{0}&1_{2}5_{0}\infty_{1}&0_{2}2_{2}\infty_{2}&1_{0}2_{0}3_{0}&0_{0}3_{1}5_{2}&0_{1}1_{1}4_{0}&2_{1}4_{2}5_{1}
\end{array}$$

point $6_{1}$:
$$\begin{array}{llllllll}
3_{0}4_{1}\infty_{1}&1_{1}5_{2}\infty_{2}&0_{0}2_{2}\infty_{0}&1_{2}2_{1}3_{2}&0_{2}3_{1}5_{0}&0_{1}1_{0}4_{2}&2_{0}4_{0}5_{1}\\

3_{1}4_{2}\infty_{2}&1_{1}5_{1}\infty_{0}&0_{0}2_{1}\infty_{1}&1_{0}2_{0}3_{2}&0_{1}3_{0}5_{0}&0_{2}1_{2}4_{1}&2_{2}4_{0}5_{2}\\

3_{0}4_{0}\infty_{2}&1_{0}5_{1}\infty_{1}&0_{2}2_{0}\infty_{0}&1_{2}2_{2}3_{1}&0_{0}3_{2}5_{0}&0_{1}1_{1}4_{1}&2_{1}4_{2}5_{2}\\

3_{2}4_{2}\infty_{1}&1_{2}5_{0}\infty_{0}&0_{1}2_{2}\infty_{2}&1_{1}2_{0}3_{1}&0_{0}3_{0}5_{2}&0_{2}1_{0}4_{0}&2_{1}4_{1}5_{1}\\

3_{2}4_{0}\infty_{0}&1_{1}5_{0}\infty_{1}&0_{2}2_{1}\infty_{2}&1_{2}2_{0}3_{0}&0_{1}3_{1}5_{2}&0_{0}1_{0}4_{1}&2_{2}4_{2}5_{1}\\

3_{0}4_{2}\infty_{0}&1_{2}5_{1}\infty_{2}&0_{1}2_{0}\infty_{1}&1_{0}2_{1}3_{1}&0_{2}3_{2}5_{2}&0_{0}1_{1}4_{0}&2_{2}4_{1}5_{0}\\

3_{2}4_{1}\infty_{2}&1_{0}5_{2}\infty_{0}&0_{2}2_{2}\infty_{1}&1_{1}2_{1}3_{0}&0_{0}3_{1}5_{1}&0_{1}1_{2}4_{0}&2_{0}4_{2}5_{0}\\

3_{1}4_{0}\infty_{1}&1_{0}5_{0}\infty_{2}&0_{1}2_{1}\infty_{0}&1_{1}2_{2}3_{2}&0_{2}3_{0}5_{1}&0_{0}1_{2}4_{2}&2_{0}4_{1}5_{2}\\

3_{1}4_{1}\infty_{0}&1_{2}5_{2}\infty_{1}&0_{0}2_{0}\infty_{2}&1_{0}2_{2}3_{0}&0_{1}3_{2}5_{1}&0_{2}1_{1}4_{2}&2_{1}4_{0}5_{0}
\end{array}$$

point $6_{2}$:
$$\begin{array}{llllllll}
3_{0}4_{2}\infty_{2}&1_{2}5_{1}\infty_{1}&0_{0}2_{1}\infty_{0}&1_{1}2_{2}3_{1}&0_{1}3_{2}5_{0}&0_{2}1_{0}4_{1}&2_{0}4_{0}5_{2}\\

3_{2}4_{2}\infty_{0}&1_{1}5_{1}\infty_{2}&0_{0}2_{0}\infty_{1}&1_{0}2_{1}3_{0}&0_{2}3_{1}5_{2}&0_{1}1_{2}4_{1}&2_{2}4_{0}5_{0}\\

3_{2}4_{1}\infty_{1}&1_{2}5_{2}\infty_{0}&0_{0}2_{2}\infty_{2}&1_{0}2_{0}3_{1}&0_{2}3_{0}5_{0}&0_{1}1_{1}4_{2}&2_{1}4_{0}5_{1}\\

3_{1}4_{1}\infty_{2}&1_{1}5_{0}\infty_{0}&0_{2}2_{1}\infty_{1}&1_{2}2_{0}3_{2}&0_{0}3_{0}5_{1}&0_{1}1_{0}4_{0}&2_{2}4_{2}5_{2}\\

3_{1}4_{2}\infty_{1}&1_{0}5_{1}\infty_{0}&0_{1}2_{1}\infty_{2}&1_{2}2_{2}3_{0}&0_{0}3_{2}5_{2}&0_{2}1_{1}4_{0}&2_{0}4_{1}5_{0}\\

3_{0}4_{0}\infty_{1}&1_{0}5_{2}\infty_{2}&0_{1}2_{0}\infty_{0}&1_{1}2_{1}3_{2}&0_{0}3_{1}5_{0}&0_{2}1_{2}4_{2}&2_{2}4_{1}5_{1}\\

3_{2}4_{0}\infty_{2}&1_{0}5_{0}\infty_{1}&0_{2}2_{2}\infty_{0}&1_{2}2_{1}3_{1}&0_{1}3_{0}5_{2}&0_{0}1_{1}4_{1}&2_{0}4_{2}5_{1}\\

3_{1}4_{0}\infty_{0}&1_{2}5_{0}\infty_{2}&0_{1}2_{2}\infty_{1}&1_{1}2_{0}3_{0}&0_{2}3_{2}5_{1}&0_{0}1_{0}4_{2}&2_{1}4_{1}5_{2}\\

3_{0}4_{1}\infty_{0}&1_{1}5_{2}\infty_{1}&0_{2}2_{0}\infty_{2}&1_{0}2_{2}3_{2}&0_{1}3_{1}5_{1}&0_{0}1_{2}4_{0}&2_{1}4_{2}5_{0}
\end{array}$$

\section{Appendix: The derived designs in Lemma \ref{gdd3_14}}

For the given GDD$(3, 4, 42)$ of type $3^{14}$  in Lemma \ref{gdd3_14}, we list all of the blocks for the derived design GDD$(2, 3, 39)$ of type $3^{13}$ at each point $x \in Z_{14} \times Z_{3}$. The thirteen blocks in  each two rows form a parallel class corresponding to the derived design, where $(i,j)$ is denoted by $i_{j}$.

point $0_{0}$:
$$\begin{array}{llllllll}
9_{2}10_{0}13_{2}&9_{1}11_{1}12_{0}&1_{0}8_{0}9_{0}&6_{0}8_{2}10_{1}&2_{0}5_{1}10_{2}&7_{2}8_{1}13_{0}&4_{1}7_{0}11_{2}\\\vspace{0.2cm}
3_{1}6_{2}11_{0}&3_{0}4_{2}13_{1}&1_{1}4_{0}5_{2}&5_{0}7_{1}12_{2}&2_{1}6_{1}12_{1}&1_{2}2_{2}3_{2}\\

9_{1}10_{0}13_{0}&9_{0}11_{2}12_{0}&1_{0}8_{1}9_{2}&6_{2}8_{2}10_{2}&2_{0}5_{2}10_{1}&7_{2}8_{0}13_{1}&4_{0}7_{0}11_{0}\\\vspace{0.2cm}
3_{1}6_{1}11_{1}&3_{2}4_{2}13_{2}&1_{2}4_{1}5_{0}&5_{1}7_{1}12_{1}&2_{1}6_{0}12_{2}&1_{1}2_{2}3_{0}\\

9_{0}10_{2}13_{2}&9_{2}11_{1}12_{2}&1_{0}8_{2}9_{1}&6_{1}8_{1}10_{1}&2_{2}5_{1}10_{0}&7_{0}8_{0}13_{0}&4_{1}7_{2}11_{0}\\\vspace{0.2cm}
3_{1}6_{0}11_{2}&3_{2}4_{0}13_{1}&1_{1}4_{2}5_{0}&5_{2}7_{1}12_{0}&2_{0}6_{2}12_{1}&1_{2}2_{1}3_{0}\\

9_{0}10_{0}13_{1}&9_{2}11_{0}12_{0}&1_{1}8_{1}9_{1}&6_{1}8_{0}10_{2}&2_{1}5_{1}10_{1}&7_{1}8_{2}13_{0}&4_{2}7_{0}11_{1}\\\vspace{0.2cm}
3_{2}6_{2}11_{2}&3_{0}4_{1}13_{2}&1_{0}4_{0}5_{0}&5_{2}7_{2}12_{2}&2_{2}6_{0}12_{1}&1_{2}2_{0}3_{1}\\
\end{array}$$

point $0_{1}$:
$$\begin{array}{llllllll}
9_{1}10_{0}13_{1}&9_{0}11_{1}12_{2}&1_{0}8_{0}9_{2}&6_{1}8_{2}10_{2}&2_{0}5_{1}10_{1}&7_{1}8_{1}13_{0}&4_{0}7_{0}11_{2}\\\vspace{0.2cm}
3_{0}6_{2}11_{0} &3_{2}4_{1}13_{2} &1_{1}4_{2}5_{2} &5_{0}7_{2}12_{0} &2_{1}6_{0}12_{1} &1_{2}2_{2}3_{1}\\

9_{0}10_{0}13_{2} &9_{2}11_{1}12_{0} &1_{0}8_{1}9_{1} &6_{1}8_{0}10_{1} &2_{1}5_{2}10_{2} &7_{0}8_{2}13_{0} &4_{1}7_{1}11_{0}\\\vspace{0.2cm}
3_{0}6_{0}11_{2} &3_{2}4_{2}13_{1} &1_{2}4_{0}5_{0} &5_{1}7_{2}12_{2} &2_{2}6_{2}12_{1} &1_{1}2_{0}3_{1}\\

9_{1}10_{1}13_{0} &9_{2}11_{0}12_{1} &1_{0}8_{2}9_{0} &6_{2}8_{0}10_{0} &2_{2}5_{1}10_{2} &7_{2}8_{1}13_{2} &4_{2}7_{1}11_{2}\\\vspace{0.2cm}
3_{0}6_{1}11_{1} &3_{1}4_{0}13_{1} &1_{1}4_{1}5_{0} &5_{2}7_{0}12_{0} &2_{0}6_{0}12_{2} &1_{2}2_{1}3_{2}\\

9_{1}10_{2}13_{2} &9_{2}11_{2}12_{2} &1_{1}8_{1}9_{0} &6_{0}8_{2}10_{0} &2_{2}5_{2}10_{1} &7_{1}8_{0}13_{1} &4_{0}7_{2}11_{0}\\\vspace{0.2cm}
3_{2}6_{2}11_{1} &3_{1}4_{1}13_{0} &1_{0}4_{2}5_{0} &5_{1}7_{0}12_{1} &2_{1}6_{1}12_{0} &1_{2}2_{0}3_{0}\\
\end{array}$$

point $0_{2}$:
$$\begin{array}{llllllll}
9_{0}10_{0}13_{0} &9_{2}11_{2}12_{0} &1_{0}8_{0}9_{1} &6_{0}8_{2}10_{2} &2_{0}5_{0}10_{1} &7_{2}8_{1}13_{1} &4_{1}7_{0}11_{0} \\\vspace{0.2cm}
3_{1}6_{2}11_{1} &3_{2}4_{0}13_{2} &1_{1}4_{2}5_{1} &5_{2}7_{1}12_{1} &2_{1}6_{1}12_{2} &1_{2}2_{2}3_{0}\\

 9_{2}10_{0}13_{1} &9_{1}11_{1}12_{2} &1_{0}8_{1}9_{0} &6_{1}8_{2}10_{1} &2_{2}5_{0}10_{2} &7_{2}8_{0}13_{2} &4_{1}7_{1}11_{2} \\\vspace{0.2cm}
 3_{1}6_{0}11_{0} &3_{2}4_{2}13_{0} &1_{2}4_{0}5_{2} &5_{1}7_{0}12_{0} &2_{1}6_{2}12_{1} &1_{1}2_{0}3_{0}\\

 9_{1}10_{2}13_{0} &9_{0}11_{0}12_{1} &1_{0}8_{2}9_{2} &6_{0}8_{1}10_{0} &2_{2}5_{1}10_{1} &7_{0}8_{0}13_{1} &4_{0}7_{2}11_{2} \\\vspace{0.2cm}
 3_{2}6_{1}11_{1} &3_{0}4_{2}13_{2} &1_{1}4_{1}5_{2} &5_{0}7_{1}12_{0} &2_{0}6_{2}12_{2} &1_{2}2_{1}3_{1}\\

 9_{0}10_{2}13_{1} &9_{2}11_{0}12_{2} &1_{2}8_{1}9_{1} &6_{0}8_{0}10_{1} &2_{0}5_{1}10_{0} &7_{0}8_{2}13_{2} &4_{2}7_{1}11_{1} \\\vspace{0.2cm}
 3_{0}6_{2}11_{2} &3_{1}4_{0}13_{0} &1_{0}4_{1}5_{0} &5_{2}7_{2}12_{0} &2_{2}6_{1}12_{1} &1_{1}2_{1}3_{2}\\
\end{array}$$

point $1_{0}$:
$$\begin{array}{llllllll}
5_{0}8_{1}12_{1} &4_{1}9_{1}12_{2} &0_{0}12_{0}13_{0} &3_{0}9_{2}13_{1} &6_{1}9_{0}11_{2} &2_{1}4_{0}13_{2} &2_{2}7_{1}8_{0} \\\vspace{0.2cm}
3_{2}8_{2}10_{2} &4_{2}7_{0}10_{1} &0_{1}6_{0}7_{2} &2_{0}5_{1}6_{2} &3_{1}5_{2}11_{0} &0_{2}10_{0}11_{1}\\

 5_{1}8_{0}12_{1} &4_{1}9_{0}12_{0} &0_{0}12_{2}13_{1} &3_{0}9_{1}13_{2} &6_{0}9_{2}11_{1} &2_{0}4_{0}13_{0} &2_{1}7_{0}8_{2} \\\vspace{0.2cm}
 3_{2}8_{1}10_{0} &4_{2}7_{2}10_{2} &0_{1}6_{1}7_{1} &2_{2}5_{2}6_{2} &3_{1}5_{0}11_{2} &0_{2}10_{1}11_{0}\\

 5_{1}8_{2}12_{2} &4_{0}9_{1}12_{0} &0_{0}12_{1}13_{2} &3_{1}9_{2}13_{0} &6_{0}9_{0}11_{0} &2_{1}4_{1}13_{1} &2_{0}7_{2}8_{1} \\\vspace{0.2cm}
 3_{2}8_{0}10_{1} &4_{2}7_{1}10_{0} &0_{1}6_{2}7_{0} &2_{2}5_{0}6_{1} &3_{0}5_{2}11_{1} &0_{2}10_{2}11_{2}\\

 5_{0}8_{2}12_{0} &4_{2}9_{1}12_{1} &0_{2}12_{2}13_{2} &3_{0}9_{0}13_{0} &6_{2}9_{2}11_{2} &2_{2}4_{0}13_{1} &2_{1}7_{1}8_{1} \\\vspace{0.2cm}
 3_{1}8_{0}10_{2} &4_{1}7_{2}10_{0} &0_{0}6_{0}7_{0} &2_{0}5_{2}6_{1} &3_{2}5_{1}11_{0} &0_{1}10_{1}11_{1}\\
\end{array}$$

point $1_{1}$:
$$\begin{array}{llllllll}
5_{0}8_{0}12_{0} &4_{1}9_{0}12_{1} &0_{0}12_{2}13_{0} &3_{0}9_{1}13_{1} &6_{1}9_{2}11_{2} &2_{0}4_{0}13_{2} &2_{2}7_{2}8_{1} \\\vspace{0.2cm}
3_{2}8_{2}10_{1} &4_{2}7_{0}10_{0} &0_{1}6_{0}7_{1} &2_{1}5_{2}6_{2} &3_{1}5_{1}11_{0} &0_{2}10_{2}11_{1}\\

 5_{0}8_{1}12_{2} &4_{0}9_{2}12_{0} &0_{0}12_{1}13_{1} &3_{0}9_{0}13_{2} &6_{0}9_{1}11_{1} &2_{1}4_{1}13_{0} &2_{0}7_{0}8_{2} \\\vspace{0.2cm}
 3_{1}8_{0}10_{1} &4_{2}7_{1}10_{2} &0_{1}6_{2}7_{2} &2_{2}5_{2}6_{1} &3_{2}5_{1}11_{2} &0_{2}10_{0}11_{0}\\

 5_{1}8_{0}12_{2} &4_{0}9_{1}12_{1} &0_{0}12_{0}13_{2} &3_{1}9_{0}13_{1} &6_{2}9_{2}11_{1} &2_{0}4_{2}13_{0} &2_{1}7_{2}8_{2} \\\vspace{0.2cm}
 3_{2}8_{1}10_{2} &4_{1}7_{1}10_{0} &0_{1}6_{1}7_{0} &2_{2}5_{0}6_{0} &3_{0}5_{2}11_{0} &0_{2}10_{1}11_{2}\\

 5_{1}8_{2}12_{0} &4_{1}9_{2}12_{2} &0_{2}12_{1}13_{2} &3_{2}9_{0}13_{0} &6_{1}9_{1}11_{0} &2_{2}4_{2}13_{1} &2_{1}7_{0}8_{1} \\\vspace{0.2cm}
 3_{0}8_{0}10_{2} &4_{0}7_{1}10_{1} &0_{0}6_{0}7_{2} &2_{0}5_{0}6_{2} &3_{1}5_{2}11_{2} &0_{1}10_{0}11_{1}\\
\end{array}$$

point $1_{2}$:
$$\begin{array}{llllllll}

 5_{0}8_{2}12_{2} &4_{1}9_{2}12_{0} &0_{0}12_{1}13_{0} &3_{0}9_{0}13_{1} &6_{0}9_{1}11_{0} &2_{2}4_{0}13_{2} &2_{0}7_{1}8_{0} \\\vspace{0.2cm}
 3_{2}8_{1}10_{1} &4_{2}7_{0}10_{2} &0_{1}6_{1}7_{2} &2_{1}5_{1}6_{2} &3_{1}5_{2}11_{1} &0_{2}10_{0}11_{2}\\

 5_{1}8_{1}12_{2} &4_{1}9_{1}12_{1} &0_{0}12_{0}13_{1} &3_{0}9_{2}13_{2} &6_{2}9_{0}11_{2} &2_{2}4_{2}13_{0} &2_{0}7_{2}8_{2} \\\vspace{0.2cm}
 3_{1}8_{0}10_{0} &4_{0}7_{0}10_{1} &0_{2}6_{1}7_{1} &2_{1}5_{0}6_{0} &3_{2}5_{2}11_{0} &0_{1}10_{2}11_{1}\\

 5_{1}8_{2}12_{1} &4_{0}9_{0}12_{0} &0_{0}12_{2}13_{2} &3_{1}9_{2}13_{1} &6_{2}9_{1}11_{1} &2_{0}4_{1}13_{0} &2_{2}7_{1}8_{1} \\\vspace{0.2cm}
 3_{0}8_{0}10_{1} &4_{2}7_{2}10_{0} &0_{1}6_{0}7_{0} &2_{1}5_{2}6_{1} &3_{2}5_{0}11_{2} &0_{2}10_{2}11_{0}\\

 5_{0}8_{0}12_{1} &4_{1}9_{0}12_{2} &0_{1}12_{0}13_{0} &3_{1}9_{1}13_{2} &6_{2}9_{2}11_{0} &2_{1}4_{2}13_{1} &2_{0}7_{0}8_{1} \\\vspace{0.2cm}
 3_{2}8_{2}10_{0} &4_{0}7_{2}10_{2} &0_{0}6_{0}7_{1} &2_{2}5_{1}6_{1} &3_{0}5_{2}11_{2} &0_{2}10_{1}11_{1}\\
\end{array}$$

point $2_{0}$:
$$\begin{array}{llllllll}

7_{0}0_{0}9_{2} &1_{1}11_{1}12_{2} &5_{2}8_{0}13_{1} &5_{0}9_{0}12_{0} &3_{0}4_{2}5_{1} &1_{2}4_{1}13_{0} &7_{1}10_{0}13_{2} \\\vspace{0.2cm}
0_{2}4_{0}8_{1} &0_{1}6_{1}12_{1} &6_{2}8_{2}11_{2} &3_{1}6_{0}7_{2} &1_{0}9_{1}10_{2} &3_{2}10_{1}11_{0}\\

 7_{1}0_{0}9_{1} &1_{1}11_{0}12_{0} &5_{0}8_{1}13_{2} &5_{2}9_{0}12_{1} &3_{1}4_{1}5_{1} &1_{2}4_{0}13_{1} &7_{0}10_{0}13_{0} \\\vspace{0.2cm}
 0_{2}4_{2}8_{2} &0_{1}6_{0}12_{2} &6_{2}8_{0}11_{1} &3_{0}6_{1}7_{2} &1_{0}9_{2}10_{1} &3_{2}10_{2}11_{2}\\

 7_{2}0_{0}9_{0} &1_{1}11_{2}12_{1} &5_{1}8_{1}13_{1} &5_{2}9_{1}12_{0} &3_{2}4_{1}5_{0} &1_{0}4_{0}13_{0} &7_{0}10_{1}13_{2} \\\vspace{0.2cm}
 0_{1}4_{2}8_{0} &0_{2}6_{2}12_{2} &6_{0}8_{2}11_{1} &3_{1}6_{1}7_{1} &1_{2}9_{2}10_{2} &3_{0}10_{0}11_{0}\\

 7_{0}0_{1}9_{0} &1_{0}11_{1}12_{1} &5_{1}8_{2}13_{0} &5_{2}9_{2}12_{2} &3_{1}4_{2}5_{0} &1_{1}4_{1}13_{1} &7_{2}10_{2}13_{2} \\\vspace{0.2cm}
 0_{0}4_{0}8_{0} &0_{2}6_{1}12_{0} &6_{2}8_{1}11_{0} &3_{2}6_{0}7_{1} &1_{2}9_{1}10_{0} &3_{0}10_{1}11_{2}\\
\end{array}$$

point $2_{1}$:
$$\begin{array}{llllllll}

 7_{0}0_{0}9_{1} &1_{0}11_{2}12_{2} &5_{1}8_{1}13_{0} &5_{0}9_{2}12_{0} &3_{1}4_{2}5_{2} &1_{2}4_{1}13_{2} &7_{2}10_{2}13_{1} \\\vspace{0.2cm}
 0_{2}4_{0}8_{0} &0_{1}6_{0}12_{1} &6_{2}8_{2}11_{1} &3_{0}6_{1}7_{1} &1_{1}9_{0}10_{1} &3_{2}10_{0}11_{0}\\

 7_{1}0_{0}9_{0} &1_{0}11_{1}12_{0} &5_{0}8_{0}13_{2} &5_{2}9_{2}12_{1} &3_{0}4_{1}5_{1} &1_{1}4_{0}13_{1} &7_{2}10_{0}13_{0} \\\vspace{0.2cm}
 0_{2}4_{2}8_{1} &0_{1}6_{2}12_{2} &6_{0}8_{2}11_{0} &3_{1}6_{1}7_{0} &1_{2}9_{1}10_{2} &3_{2}10_{1}11_{2}\\

 7_{2}0_{0}9_{2} &1_{0}11_{0}12_{1} &5_{1}8_{2}13_{2} &5_{0}9_{0}12_{2} &3_{2}4_{1}5_{2} &1_{2}4_{2}13_{1} &7_{1}10_{1}13_{0} \\\vspace{0.2cm}
 0_{1}4_{0}8_{1} &0_{2}6_{0}12_{0} &6_{1}8_{0}11_{1} &3_{0}6_{2}7_{0} &1_{1}9_{1}10_{0} &3_{1}10_{2}11_{2}\\

 7_{0}0_{2}9_{0} &1_{1}11_{0}12_{2} &5_{2}8_{0}13_{0} &5_{0}9_{1}12_{1} &3_{2}4_{2}5_{1} &1_{0}4_{1}13_{1} &7_{1}10_{2}13_{2} \\\vspace{0.2cm}
 0_{0}4_{0}8_{2} &0_{1}6_{1}12_{0} &6_{2}8_{1}11_{2} &3_{0}6_{0}7_{2} &1_{2}9_{2}10_{1} &3_{1}10_{0}11_{1}\\
\end{array}$$

point $2_{2}$:
$$\begin{array}{llllllll}
 7_{0}0_{0}9_{0} &1_{2}11_{0}12_{2} &5_{1}8_{0}13_{0} &5_{0}9_{1}12_{0} &3_{1}4_{1}5_{2} &1_{1}4_{2}13_{1} &7_{2}10_{0}13_{2} \\\vspace{0.2cm}
 0_{2}4_{0}8_{2} &0_{1}6_{2}12_{1} &6_{1}8_{1}11_{2} &3_{0}6_{0}7_{1} &1_{0}9_{2}10_{2} &3_{2}10_{1}11_{1}\\

 7_{1}0_{0}9_{2} &1_{2}11_{2}12_{0} &5_{0}8_{0}13_{1} &5_{2}9_{1}12_{1} &3_{0}4_{0}5_{1} &1_{0}4_{2}13_{2} &7_{0}10_{1}13_{0} \\\vspace{0.2cm}
 0_{2}4_{1}8_{1} &0_{1}6_{1}12_{2} &6_{2}8_{2}11_{0} &3_{2}6_{0}7_{2} &1_{1}9_{0}10_{0} &3_{1}10_{2}11_{1}\\

 7_{2}0_{0}9_{1} &1_{2}11_{1}12_{1} &5_{2}8_{0}13_{2} &5_{0}9_{2}12_{2} &3_{1}4_{2}5_{1} &1_{1}4_{0}13_{0} &7_{1}10_{2}13_{1} \\\vspace{0.2cm}
 0_{1}4_{1}8_{2} &0_{2}6_{2}12_{0} &6_{0}8_{1}11_{0} &3_{0}6_{1}7_{0} &1_{0}9_{0}10_{1} &3_{2}10_{0}11_{2}\\

 7_{1}0_{1}9_{0} &1_{1}11_{1}12_{0} &5_{0}8_{2}13_{2} &5_{1}9_{1}12_{2} &3_{0}4_{2}5_{2} &1_{0}4_{1}13_{0} &7_{2}10_{1}13_{1} \\\vspace{0.2cm}
 0_{0}4_{0}8_{1} &0_{2}6_{1}12_{1} &6_{2}8_{0}11_{2} &3_{1}6_{0}7_{0} &1_{2}9_{2}10_{0} &3_{2}10_{2}11_{0}\\
\end{array}$$

point $3_{0}$:
$$\begin{array}{llllllll}
 2_{0}12_{0}13_{2} &12_{2}10_{1}5_{1} &0_{0}6_{0}11_{0} &0_{1}7_{2}10_{0} &0_{2}1_{0}2_{1} &1_{1}5_{0}11_{2} &8_{2}11_{1}13_{0} \\\vspace{0.2cm}
 1_{2}4_{0}6_{1} &4_{2}9_{2}10_{2} &6_{2}9_{0}12_{1} &2_{2}8_{0}9_{1} &4_{1}7_{1}8_{1} &5_{2}7_{0}13_{1}\\

 2_{0}12_{1}13_{1} &12_{0}10_{1}5_{0} &0_{0}6_{2}11_{1} &0_{2}7_{1}10_{0} &0_{1}1_{0}2_{2} &1_{2}5_{1}11_{0} &8_{1}11_{2}13_{0} \\\vspace{0.2cm}
 1_{1}4_{2}6_{0} &4_{0}9_{1}10_{2} &6_{1}9_{0}12_{2} &2_{1}8_{0}9_{2} &4_{1}7_{0}8_{2} &5_{2}7_{2}13_{2}\\

 2_{0}12_{2}13_{0} &12_{1}10_{1}5_{2} &0_{0}6_{1}11_{2} &0_{2}7_{2}10_{2} &0_{1}1_{1}2_{1} &1_{0}5_{0}11_{0} &8_{1}11_{1}13_{1} \\\vspace{0.2cm}
 1_{2}4_{1}6_{0} &4_{0}9_{0}10_{0} &6_{2}9_{1}12_{0} &2_{2}8_{2}9_{2} &4_{2}7_{1}8_{0} &5_{1}7_{0}13_{2}\\

 2_{2}12_{1}13_{0} &12_{2}10_{2}5_{2} &0_{2}6_{2}11_{2} &0_{0}7_{0}10_{0} &0_{1}1_{2}2_{0} &1_{1}5_{1}11_{1} &8_{2}11_{0}13_{1} \\\vspace{0.2cm}
 1_{0}4_{2}6_{1} &4_{0}9_{2}10_{1} &6_{0}9_{0}12_{0} &2_{1}8_{1}9_{1} &4_{1}7_{2}8_{0} &5_{0}7_{1}13_{2}\\
\end{array}$$

point $3_{1}$:
$$\begin{array}{llllllll}
 2_{1}12_{0}13_{2} &12_{2}10_{2}5_{1} &0_{0}6_{2}11_{0} &0_{1}7_{0}10_{1} &0_{2}1_{0}2_{0} &1_{1}5_{0}11_{1} &8_{0}11_{2}13_{0} \\\vspace{0.2cm}
 1_{2}4_{0}6_{0} &4_{1}9_{1}10_{0} &6_{1}9_{0}12_{1} &2_{2}8_{1}9_{2} &4_{2}7_{1}8_{2} &5_{2}7_{2}13_{1}\\

 2_{2}12_{0}13_{0} &12_{1}10_{1}5_{1} &0_{0}6_{1}11_{1} &0_{2}7_{0}10_{0} &0_{1}1_{0}2_{1} &1_{2}5_{0}11_{0} &8_{2}11_{2}13_{1} \\\vspace{0.2cm}
 1_{1}4_{1}6_{0} &4_{0}9_{0}10_{2} &6_{2}9_{1}12_{2} &2_{0}8_{0}9_{2} &4_{2}7_{2}8_{1} &5_{2}7_{1}13_{2}\\

 2_{1}12_{2}13_{0} &12_{1}10_{2}5_{2} &0_{0}6_{0}11_{2} &0_{1}7_{1}10_{0} &0_{2}1_{1}2_{2} &1_{0}5_{1}11_{1} &8_{1}11_{0}13_{1} \\\vspace{0.2cm}
 1_{2}4_{1}6_{2} &4_{2}9_{2}10_{1} &6_{1}9_{1}12_{0} &2_{0}8_{2}9_{0} &4_{0}7_{2}8_{0} &5_{0}7_{0}13_{2}\\

 2_{1}12_{1}13_{1} &12_{0}10_{2}5_{0} &0_{2}6_{0}11_{0} &0_{0}7_{2}10_{0} &0_{1}1_{2}2_{2} &1_{1}5_{2}11_{2} &8_{2}11_{1}13_{2} \\\vspace{0.2cm}
 1_{0}4_{0}6_{2} &4_{1}9_{0}10_{1} &6_{1}9_{2}12_{2} &2_{0}8_{1}9_{1} &4_{2}7_{0}8_{0} &5_{1}7_{1}13_{0}\\
\end{array}$$

point $3_{2}$:
$$\begin{array}{llllllll}
 2_{2}12_{0}13_{2} &12_{1}10_{0}5_{2} &0_{0}6_{1}11_{0} &0_{2}7_{0}10_{2} &0_{1}1_{0}2_{0} &1_{2}5_{0}11_{2} &8_{0}11_{1}13_{0} \\\vspace{0.2cm}
 1_{1}4_{0}6_{0} &4_{1}9_{2}10_{1} &6_{2}9_{0}12_{2} &2_{1}8_{2}9_{1} &4_{2}7_{1}8_{1} &5_{1}7_{2}13_{1}\\

 2_{1}12_{0}13_{1} &12_{1}10_{1}5_{0} &0_{0}6_{0}11_{1} &0_{1}7_{0}10_{0} &0_{2}1_{0}2_{2} &1_{2}5_{2}11_{0} &8_{2}11_{2}13_{0} \\\vspace{0.2cm}
 1_{1}4_{1}6_{2} &4_{0}9_{2}10_{2} &6_{1}9_{1}12_{2} &2_{0}8_{1}9_{0} &4_{2}7_{2}8_{0} &5_{1}7_{1}13_{2}\\

 2_{0}12_{0}13_{0} &12_{1}10_{2}5_{1} &0_{0}6_{2}11_{2} &0_{2}7_{2}10_{0} &0_{1}1_{1}2_{2} &1_{0}5_{0}11_{1} &8_{2}11_{0}13_{2} \\\vspace{0.2cm}
 1_{2}4_{1}6_{1} &4_{2}9_{1}10_{1} &6_{0}9_{2}12_{2} &2_{1}8_{0}9_{0} &4_{0}7_{0}8_{1} &5_{2}7_{1}13_{1}\\

 2_{1}12_{2}13_{2} &12_{0}10_{1}5_{1} &0_{1}6_{2}11_{1} &0_{0}7_{1}10_{0} &0_{2}1_{2}2_{0} &1_{1}5_{0}11_{0} &8_{1}11_{2}13_{1} \\\vspace{0.2cm}
 1_{0}4_{0}6_{1} &4_{1}9_{1}10_{2} &6_{0}9_{0}12_{1} &2_{2}8_{0}9_{2} &4_{2}7_{0}8_{2} &5_{2}7_{2}13_{0}\\
\end{array}$$

point $4_{0}$:
$$\begin{array}{llllllll}
 8_{0}10_{0}13_{2} &1_{0}9_{1}12_{0} &1_{2}7_{2}10_{2} &0_{0}7_{0}11_{0} &5_{0}6_{2}7_{1} &2_{2}6_{0}10_{1} &1_{1}3_{1}6_{1} \\\vspace{0.2cm}
 3_{0}11_{1}12_{2} &2_{1}9_{0}11_{2} &5_{1}12_{1}13_{1} &0_{1}3_{2}13_{0} &0_{2}2_{0}8_{1} &5_{2}8_{2}9_{2}\\

 8_{2}10_{1}13_{0} &1_{1}9_{0}12_{2} &1_{0}7_{1}10_{2} &0_{0}7_{2}11_{1} &5_{2}6_{1}7_{0} &2_{0}6_{0}10_{0} &1_{2}3_{2}6_{2} \\\vspace{0.2cm}
 3_{0}11_{2}12_{1} &2_{2}9_{1}11_{0} &5_{1}12_{0}13_{2} &0_{1}3_{1}13_{1} &0_{2}2_{1}8_{0} &5_{0}8_{1}9_{2}\\

 8_{1}10_{2}13_{1} &1_{2}9_{2}12_{1} &1_{0}7_{2}10_{1} &0_{0}7_{1}11_{2} &5_{1}6_{2}7_{0} &2_{2}6_{1}10_{0} &1_{1}3_{2}6_{0} \\\vspace{0.2cm}
 3_{0}11_{0}12_{0} &2_{1}9_{1}11_{1} &5_{2}12_{2}13_{2} &0_{2}3_{1}13_{0} &0_{1}2_{0}8_{2} &5_{0}8_{0}9_{0}\\

 8_{2}10_{2}13_{2} &1_{1}9_{2}12_{0} &1_{0}7_{0}10_{0} &0_{2}7_{2}11_{2} &5_{1}6_{1}7_{1} &2_{0}6_{2}10_{1} &1_{2}3_{1}6_{0} \\\vspace{0.2cm}
 3_{2}11_{0}12_{1} &2_{2}9_{0}11_{1} &5_{0}12_{2}13_{1} &0_{0}3_{0}13_{0} &0_{1}2_{1}8_{1} &5_{2}8_{0}9_{1}\\
\end{array}$$

point $4_{1}$:
$$\begin{array}{llllllll}
 8_{1}10_{1}13_{1} &1_{1}9_{1}12_{0} &1_{2}7_{0}10_{0} &0_{0}7_{2}11_{0} &5_{2}6_{2}7_{1} &2_{2}6_{1}10_{2} &1_{0}3_{2}6_{0} \\\vspace{0.2cm}
 3_{0}11_{1}12_{1} &2_{1}9_{2}11_{2} &5_{1}12_{2}13_{2} &0_{1}3_{1}13_{0} &0_{2}2_{0}8_{0} &5_{0}8_{2}9_{0}\\

 8_{1}10_{2}13_{0} &1_{1}9_{0}12_{1} &1_{2}7_{2}10_{1} &0_{0}7_{1}11_{1} &5_{0}6_{2}7_{0} &2_{2}6_{0}10_{0} &1_{0}3_{1}6_{1} \\\vspace{0.2cm}
 3_{0}11_{2}12_{0} &2_{0}9_{2}11_{0} &5_{2}12_{2}13_{1} &0_{1}3_{2}13_{2} &0_{2}2_{1}8_{2} &5_{1}8_{0}9_{1}\\

 8_{1}10_{0}13_{2} &1_{1}9_{2}12_{2} &1_{0}7_{1}10_{1} &0_{0}7_{0}11_{2} &5_{1}6_{2}7_{2} &2_{0}6_{0}10_{2} &1_{2}3_{2}6_{1} \\\vspace{0.2cm}
 3_{1}11_{1}12_{0} &2_{1}9_{1}11_{0} &5_{0}12_{1}13_{1} &0_{2}3_{0}13_{0} &0_{1}2_{2}8_{2} &5_{2}8_{0}9_{0}\\

 8_{2}10_{2}13_{1} &1_{0}9_{0}12_{0} &1_{1}7_{0}10_{1} &0_{1}7_{2}11_{2} &5_{1}6_{0}7_{1} &2_{1}6_{1}10_{0} &1_{2}3_{1}6_{2} \\\vspace{0.2cm}
 3_{0}11_{0}12_{2} &2_{0}9_{1}11_{1} &5_{2}12_{1}13_{2} &0_{0}3_{2}13_{0} &0_{2}2_{2}8_{1} &5_{0}8_{0}9_{2}\\
\end{array}$$

point $4_{2}$:
$$\begin{array}{llllllll}
 8_{0}10_{1}13_{2} &1_{1}9_{1}12_{2} &1_{0}7_{2}10_{2} &0_{0}7_{1}11_{0} &5_{2}6_{2}7_{0} &2_{1}6_{0}10_{0} &1_{2}3_{1}6_{1} \\\vspace{0.2cm}
 3_{0}11_{1}12_{0} &2_{2}9_{0}11_{2} &5_{0}12_{1}13_{0} &0_{1}3_{2}13_{1} &0_{2}2_{0}8_{2} &5_{1}8_{1}9_{2}\\

 8_{2}10_{2}13_{0} &1_{2}9_{1}12_{0} &1_{0}7_{1}10_{0} &0_{0}7_{0}11_{1} &5_{0}6_{2}7_{2} &2_{0}6_{0}10_{1} &1_{1}3_{2}6_{1} \\\vspace{0.2cm}
 3_{0}11_{2}12_{2} &2_{2}9_{2}11_{0} &5_{2}12_{1}13_{1} &0_{1}3_{1}13_{2} &0_{2}2_{1}8_{1} &5_{1}8_{0}9_{0}\\

 8_{2}10_{1}13_{1} &1_{2}9_{2}12_{2} &1_{1}7_{0}10_{0} &0_{0}7_{2}11_{2} &5_{1}6_{2}7_{1} &2_{1}6_{1}10_{2} &1_{0}3_{1}6_{0} \\\vspace{0.2cm}
 3_{0}11_{0}12_{1} &2_{2}9_{1}11_{1} &5_{2}12_{0}13_{2} &0_{2}3_{2}13_{0} &0_{1}2_{0}8_{0} &5_{0}8_{1}9_{0}\\

 8_{1}10_{0}13_{1} &1_{0}9_{1}12_{1} &1_{2}7_{1}10_{1} &0_{1}7_{0}11_{0} &5_{1}6_{1}7_{2} &2_{0}6_{2}10_{2} &1_{1}3_{0}6_{0} \\\vspace{0.2cm}
 3_{2}11_{2}12_{0} &2_{1}9_{2}11_{1} &5_{0}12_{2}13_{2} &0_{0}3_{1}13_{0} &0_{2}2_{2}8_{0} &5_{2}8_{2}9_{0}\\
\end{array}$$

point $5_{0}$:
$$\begin{array}{llllllll}
 4_{2}10_{2}11_{1} &3_{1}6_{2}8_{0} &6_{0}11_{0}12_{2} &6_{1}9_{1}10_{0} &2_{2}8_{1}13_{0} &2_{0}9_{2}12_{1} &2_{1}3_{2}4_{0}\\\vspace{0.2cm}
 3_{0}7_{2}13_{1} &1_{0}10_{1}13_{2} &0_{0}7_{0}12_{0} &0_{2}8_{2}11_{2} &1_{2}7_{1}9_{0} &0_{1}1_{1}4_{1}\\

 4_{0}10_{1}11_{2} &3_{2}6_{2}8_{1} &6_{1}11_{1}12_{2} &6_{0}9_{1}10_{2} &2_{2}8_{0}13_{1} &2_{0}9_{0}12_{0} &2_{1}3_{1}4_{1} \\\vspace{0.2cm}
 3_{0}7_{1}13_{2} &1_{0}10_{0}13_{0} &0_{0}7_{2}12_{1} &0_{1}8_{2}11_{0} &1_{1}7_{0}9_{2} &0_{2}1_{2}4_{2}\\

 4_{1}10_{0}11_{0} &3_{0}6_{0}8_{1} &6_{2}11_{1}12_{1} &6_{1}9_{0}10_{2} &2_{2}8_{2}13_{2} &2_{1}9_{2}12_{0} &2_{0}3_{1}4_{2} \\\vspace{0.2cm}
 3_{2}7_{0}13_{1} &1_{2}10_{1}13_{0} &0_{0}7_{1}12_{2} &0_{1}8_{0}11_{2} &1_{0}7_{2}9_{1} &0_{2}1_{1}4_{0}\\

 4_{1}10_{1}11_{1} &3_{0}6_{2}8_{2} &6_{1}11_{2}12_{0} &6_{0}9_{2}10_{0} &2_{1}8_{1}13_{1} &2_{0}9_{1}12_{2} &2_{2}3_{1}4_{0} \\\vspace{0.2cm}
 3_{2}7_{1}13_{0} &1_{2}10_{2}13_{2} &0_{2}7_{0}12_{1} &0_{0}8_{0}11_{0} &1_{1}7_{2}9_{0} &0_{1}1_{0}4_{2}\\
\end{array}$$

point $5_{1}$:
$$\begin{array}{llllllll}
 4_{0}10_{2}11_{2} &3_{2}6_{2}8_{0} &6_{0}11_{1}12_{2} &6_{1}9_{0}10_{0} &2_{1}8_{2}13_{2} &2_{0}9_{1}12_{1} &2_{2}3_{1}4_{2} \\\vspace{0.2cm}3_{0}7_{1}13_{1} &1_{1}10_{1}13_{0} &0_{0}7_{2}12_{0} &0_{1}8_{1}11_{0} &1_{0}7_{0}9_{2} &0_{2}1_{2}4_{1}\\

 4_{0}10_{0}11_{0} &3_{1}6_{0}8_{1} &6_{1}11_{2}12_{2} &6_{2}9_{1}10_{2} &2_{2}8_{2}13_{1} &2_{0}9_{2}12_{0} &2_{1}3_{0}4_{1} \\\vspace{0.2cm}3_{2}7_{0}13_{0} &1_{2}10_{1}13_{2} &0_{0}7_{1}12_{1} &0_{1}8_{0}11_{1} &1_{0}7_{2}9_{0} &0_{2}1_{1}4_{2}\\

 4_{2}10_{1}11_{2} &3_{1}6_{2}8_{2} &6_{1}11_{1}12_{1} &6_{0}9_{0}10_{2} &2_{0}8_{1}13_{1} &2_{1}9_{1}12_{0} &2_{2}3_{0}4_{0} \\\vspace{0.2cm}3_{2}7_{1}13_{2} &1_{2}10_{0}13_{0} &0_{0}7_{0}12_{2} &0_{2}8_{0}11_{0} &1_{1}7_{2}9_{2} &0_{1}1_{0}4_{1}\\

 4_{2}10_{2}11_{0} &3_{0}6_{0}8_{0} &6_{2}11_{2}12_{1} &6_{1}9_{1}10_{1} &2_{0}8_{2}13_{0} &2_{1}9_{2}12_{2} &2_{2}3_{2}4_{1} \\\vspace{0.2cm}3_{1}7_{2}13_{2} &1_{1}10_{0}13_{1} &0_{1}7_{1}12_{0} &0_{0}8_{1}11_{1} &1_{2}7_{0}9_{0} &0_{2}1_{0}4_{0}\\
\end{array}$$

point $5_{2}$:
$$\begin{array}{llllllll}
 4_{0}10_{2}11_{1} &3_{0}6_{2}8_{0} &6_{0}11_{2}12_{2} &6_{1}9_{2}10_{0} &2_{1}8_{1}13_{2} &2_{0}9_{0}12_{1} &2_{2}3_{1}4_{1} \\\vspace{0.2cm}3_{2}7_{2}13_{0} &1_{2}10_{1}13_{1} &0_{0}7_{1}12_{0} &0_{2}8_{2}11_{0} &1_{0}7_{0}9_{1} &0_{1}1_{1}4_{2}\\

 4_{1}10_{0}11_{1} &3_{0}6_{0}8_{2} &6_{1}11_{0}12_{2} &6_{2}9_{2}10_{1} &2_{2}8_{0}13_{2} &2_{1}9_{0}12_{0} &2_{0}3_{1}4_{0} \\\vspace{0.2cm}3_{2}7_{1}13_{1} &1_{2}10_{2}13_{0} &0_{0}7_{0}12_{1} &0_{1}8_{1}11_{2} &1_{1}7_{2}9_{1} &0_{2}1_{0}4_{2}\\

 4_{2}10_{1}11_{1} &3_{1}6_{2}8_{1} &6_{0}11_{0}12_{0} &6_{1}9_{1}10_{2} &2_{0}8_{2}13_{2} &2_{1}9_{2}12_{1} &2_{2}3_{2}4_{0} \\\vspace{0.2cm}3_{0}7_{0}13_{1} &1_{1}10_{0}13_{0} &0_{0}7_{2}12_{2} &0_{2}8_{0}11_{2} &1_{0}7_{1}9_{0} &0_{1}1_{2}4_{1}\\

 4_{1}10_{1}11_{2} &3_{0}6_{1}8_{1} &6_{2}11_{0}12_{1} &6_{0}9_{0}10_{0} &2_{2}8_{2}13_{0} &2_{1}9_{1}12_{2} &2_{0}3_{2}4_{2} \\\vspace{0.2cm}3_{1}7_{1}13_{2} &1_{1}10_{2}13_{1} &0_{1}7_{0}12_{0} &0_{0}8_{0}11_{1} &1_{0}7_{2}9_{2} &0_{2}1_{2}4_{0}\\
\end{array}$$

point $6_{0}$:
$$\begin{array}{llllllll}
 8_{1}4_{2}12_{2} &3_{0}5_{1}8_{0} &4_{0}11_{2}13_{1} &3_{1}9_{1}12_{1} &2_{1}9_{0}13_{2} &7_{2}8_{2}9_{2} &1_{2}10_{1}12_{0} \\\vspace{0.2cm}0_{0}5_{0}13_{0} &0_{1}3_{2}11_{0} &0_{2}1_{0}7_{1} &1_{1}2_{0}5_{2} &2_{2}4_{1}10_{0} &7_{0}10_{2}11_{1}\\

 8_{2}4_{1}12_{0} &3_{0}5_{0}8_{1} &4_{0}11_{0}13_{2} &3_{2}9_{0}12_{1} &2_{1}9_{2}13_{0} &7_{2}8_{0}9_{1} &1_{0}10_{1}12_{2} \\\vspace{0.2cm}0_{0}5_{2}13_{1} &0_{1}3_{1}11_{1} &0_{2}1_{1}7_{0} &1_{2}2_{0}5_{1} &2_{2}4_{2}10_{2} &7_{1}10_{0}11_{2}\\

 8_{0}4_{2}12_{0} &3_{1}5_{0}8_{2} &4_{1}11_{0}13_{0} &3_{2}9_{2}12_{2} &2_{1}9_{1}13_{1} &7_{2}8_{1}9_{0} &1_{0}10_{2}12_{1} \\\vspace{0.2cm}0_{0}5_{1}13_{2} &0_{2}3_{0}11_{1} &0_{1}1_{1}7_{1} &1_{2}2_{2}5_{2} &2_{0}4_{0}10_{0} &7_{0}10_{1}11_{2}\\

 8_{2}4_{2}12_{1} &3_{1}5_{1}8_{1} &4_{0}11_{1}13_{0} &3_{2}9_{1}12_{0} &2_{2}9_{0}13_{1} &7_{1}8_{0}9_{2} &1_{1}10_{0}12_{2} \\\vspace{0.2cm}0_{2}5_{2}13_{2} &0_{0}3_{0}11_{0} &0_{1}1_{2}7_{0} &1_{0}2_{0}5_{0} &2_{1}4_{1}10_{1} &7_{2}10_{2}11_{2}\\
\end{array}%
$$

\section{Appendix: The base blocks and derived designs in Lemma \ref{rdsqs28}}

The following is the base blocks under modulo $(7, -)$ in Lemma \ref{rdsqs28}, where $(a, i)$ is denoted by $a_{i}$:

$0_{0}1_{0}2_{1}4_{1}~~~0_{0}6_{0}4_{1}5_{1}~~~0_{0}2_{0}3_{1}4_{1}~~~0_{0}5_{0}4_{1}6_{1}~~~0_{0}3_{0}1_{1}4_{1}~~~0_{0}1_{0}2_{0}4_{0}~~~0_{0}1_{0}5_{0}3_{1}$

$0_{0}2_{1}3_{1}5_{1}~~~0_{0}3_{0}2_{1}6_{1}~~~2_{1}4_{1}5_{1}6_{1}~~~0_{0}3_{0}0_{1}3_{1}~~~0_{0}1_{0}0_{1}1_{1}~~~0_{0}2_{0}0_{1}2_{1}$

$0_{2}1_{2}5_{2}4_{3}~~~0_{2}6_{2}4_{3}5_{3}~~~0_{2}2_{2}1_{3}4_{3}~~~0_{2}3_{3}4_{3}6_{3}~~~0_{2}1_{2}2_{2}4_{2}~~~0_{2}1_{2}2_{3}3_{3}~~~0_{2}2_{2}3_{3}5_{3}$

$0_{2}3_{2}1_{3}5_{3}~~~0_{2}3_{2}2_{3}4_{3}~~~2_{3}4_{3}5_{3}6_{3}~~~0_{2}3_{2}0_{3}3_{3}~~~0_{2}1_{2}0_{3}1_{3}~~~0_{2}2_{2}0_{3}2_{3}$

$0_{0}1_{1}0_{2}1_{3}~~~0_{0}1_{1}1_{2}2_{3}~~~0_{0}1_{1}2_{2}3_{3}~~~0_{0}1_{1}3_{2}4_{3}~~~0_{0}1_{1}4_{2}5_{3}~~~0_{0}1_{1}5_{2}6_{3}~~~0_{0}1_{1}6_{2}0_{3}$

$0_{0}2_{1}0_{2}2_{3}~~~0_{0}2_{1}1_{2}3_{3}~~~0_{0}2_{1}2_{2}4_{3}~~~0_{0}2_{1}3_{2}5_{3}~~~0_{0}2_{1}4_{2}6_{3}~~~0_{0}2_{1}5_{2}0_{3}~~~0_{0}2_{1}6_{2}1_{3}$

$0_{0}4_{1}0_{2}4_{3}~~~0_{0}4_{1}1_{2}5_{3}~~~0_{0}4_{1}2_{2}6_{3}~~~0_{0}4_{1}3_{2}0_{3}~~~0_{0}4_{1}4_{2}1_{3}~~~0_{0}4_{1}5_{2}2_{3}~~~0_{0}4_{1}6_{2}3_{3}$

$0_{0}3_{1}0_{2}6_{3}~~~0_{0}3_{1}1_{2}0_{3}~~~0_{0}3_{1}2_{2}1_{3}~~~0_{0}3_{1}3_{2}2_{3}~~~0_{0}3_{1}4_{2}3_{3}~~~0_{0}3_{1}5_{2}4_{3}~~~0_{0}3_{1}6_{2}5_{3}$

$0_{0}5_{1}0_{2}3_{3}~~~0_{0}5_{1}1_{2}4_{3}~~~0_{0}5_{1}2_{2}5_{3}~~~0_{0}5_{1}3_{2}6_{3}~~~0_{0}5_{1}4_{2}0_{3}~~~0_{0}5_{1}5_{2}1_{3}~~~0_{0}5_{1}6_{2}2_{3}$

$0_{0}6_{1}0_{2}5_{3}~~~0_{0}6_{1}1_{2}6_{3}~~~0_{0}6_{1}2_{2}0_{3}~~~0_{0}6_{1}3_{2}1_{3}~~~0_{0}6_{1}4_{2}2_{3}~~~0_{0}6_{1}5_{2}3_{3}~~~0_{0}6_{1}6_{2}4_{3}$

$0_{0}0_{1}0_{2}0_{3}~~~0_{0}0_{1}1_{2}3_{2}~~~0_{0}0_{1}2_{2}6_{2}~~~0_{0}0_{1}4_{2}5_{2}~~~0_{0}0_{1}1_{3}3_{3}~~~0_{0}0_{1}2_{3}6_{3}~~~0_{0}0_{1}4_{3}5_{3}$

$1_{0}3_{0}0_{2}0_{3}~~~1_{0}3_{0}1_{2}3_{2}~~~1_{0}3_{0}2_{2}6_{2}~~~1_{0}3_{0}4_{2}5_{2}~~~1_{0}3_{0}1_{3}3_{3}~~~1_{0}3_{0}2_{3}6_{3}~~~1_{0}3_{0}4_{3}5_{3}$

$2_{0}6_{0}0_{2}0_{3}~~~2_{0}6_{0}1_{2}3_{2}~~~2_{0}6_{0}2_{2}6_{2}~~~2_{0}6_{0}4_{2}5_{2}~~~2_{0}6_{0}1_{3}3_{3}~~~2_{0}6_{0}2_{3}6_{3}~~~2_{0}6_{0}4_{3}5_{3}$

$4_{0}5_{0}0_{2}0_{3}~~~4_{0}5_{0}1_{2}3_{2}~~~4_{0}5_{0}2_{2}6_{2}~~~4_{0}5_{0}4_{2}5_{2}~~~4_{0}5_{0}1_{3}3_{3}~~~4_{0}5_{0}2_{3}6_{3}~~~4_{0}5_{0}4_{3}5_{3}$

$1_{1}3_{1}0_{2}0_{3}~~~1_{1}3_{1}1_{2}3_{2}~~~1_{1}3_{1}2_{2}6_{2}~~~1_{1}3_{1}4_{2}5_{2}~~~1_{1}3_{1}1_{3}3_{3}~~~1_{1}3_{1}2_{3}6_{3}~~~1_{1}3_{1}4_{3}5_{3}$

$2_{1}6_{1}0_{2}0_{3}~~~2_{1}6_{1}1_{2}3_{2}~~~2_{1}6_{1}2_{2}6_{2}~~~2_{1}6_{1}4_{2}5_{2}~~~2_{1}6_{1}1_{3}3_{3}~~~2_{1}6_{1}2_{3}6_{3}~~~2_{1}6_{1}4_{3}5_{3}$

$4_{1}5_{1}0_{2}0_{3}~~~4_{1}5_{1}1_{2}3_{2}~~~4_{1}5_{1}2_{2}6_{2}~~~4_{1}5_{1}4_{2}5_{2}~~~4_{1}5_{1}1_{3}3_{3}~~~4_{1}5_{1}2_{3}6_{3}~~~4_{1}5_{1}4_{3}5_{3}$
\vspace{0.2cm}

Next we list the block sets of the multiset of the derived design at the points $0_{0},0_{1},0_{2},0_{3}$ in Lemma \ref{rdsqs28}, where the underlined blocks form a parallel class, which occurs three times, while the other blocks occurs two times;  moreover, the blocks in each row form a  corresponding parallel class.

point $0_{0}$:
$$\begin{array}{lllllllll}
\underline{3_{3}5_{2}6_{1}}&0_{1}4_{3}5_{3}&0_{3}1_{1}6_{2}&2_{2}3_{2}4_{0}&1_{3}2_{3}5_{0}&1_{0}2_{1}4_{1}&2_{0}3_{0}5_{1}&0_{2}3_{1}6_{3}&1_{2}4_{2}6_{0}\\[2pt]

3_{3}5_{2}6_{1}&3_{0}5_{3}6_{3}&2_{3}3_{1}3_{2}&0_{2}1_{1}1_{3}&0_{1}2_{0}2_{1}&1_{2}2_{2}5_{0}&1_{0}4_{2}6_{2}&4_{1}5_{1}6_{0}&0_{3}4_{0}4_{3}\\[2pt]

3_{3}5_{2}6_{1}&0_{1}1_{0}1_{1}&1_{3}2_{2}3_{1}&1_{2}4_{1}5_{3}&2_{0}5_{0}6_{0}&2_{3}5_{1}6_{2}&0_{3}4_{0}4_{3}&2_{1}4_{2}6_{3}&0_{2}3_{0}3_{2}\\[2pt]

\underline{1_{3}4_{3}6_{0}}&1_{2}6_{1}6_{3}&0_{3}5_{0}5_{3}&3_{3}4_{1}6_{2}&0_{2}2_{1}2_{3}&1_{0}2_{2}5_{2}&1_{1}4_{0}5_{1}&0_{1}3_{0}3_{1}&2_{0}3_{2}4_{2}\\[2pt]

1_{3}4_{3}6_{0}&0_{2}4_{0}4_{2}&2_{2}5_{1}5_{3}&0_{3}2_{1}5_{2}&4_{1}5_{0}6_{1}&1_{1}1_{2}2_{3}&2_{0}6_{2}6_{3}&0_{1}3_{0}3_{1}&1_{0}3_{2}3_{3}\\[2pt]

1_{3}4_{3}6_{0}&1_{0}2_{1}4_{1}&4_{0}5_{2}5_{3}&2_{3}4_{2}6_{1}&1_{1}2_{2}3_{3}&2_{0}6_{2}6_{3}&0_{3}1_{2}3_{1}&0_{1}5_{0}5_{1}&0_{2}3_{0}3_{2}\\[2pt]

\underline{1_{2}4_{0}6_{2}}&1_{3}3_{2}6_{1}&1_{0}3_{1}5_{0}&0_{3}6_{0}6_{3}&2_{2}5_{1}5_{3}&0_{2}2_{1}2_{3}&2_{0}3_{3}4_{3}&1_{1}3_{0}4_{1}&0_{1}4_{2}5_{2}\\[2pt]

1_{2}4_{0}6_{2}&0_{3}4_{2}5_{1}&1_{3}2_{0}5_{3}&2_{1}3_{0}6_{1}&0_{1}2_{3}6_{3}&1_{0}3_{1}5_{0}&0_{2}4_{1}4_{3}&1_{1}2_{2}3_{3}&3_{2}5_{2}6_{0}\\[2pt]

1_{2}4_{0}6_{2}&0_{3}4_{2}5_{1}&2_{3}4_{1}5_{2}&0_{2}5_{3}6_{1}&1_{0}3_{0}6_{0}&1_{3}2_{2}3_{1}&3_{3}5_{0}6_{3}&0_{1}2_{0}2_{1}&1_{1}3_{2}4_{3}\\[2pt]

 \end{array}$$
$$\begin{array}{lllllllll}

\underline{2_{0}3_{1}4_{1}}&1_{2}4_{2}6_{0}&1_{0}4_{3}6_{3}&0_{3}2_{2}6_{1}&0_{2}1_{1}1_{3}&2_{1}3_{2}5_{3}&2_{3}3_{3}4_{0}&3_{0}5_{2}6_{2}&0_{1}5_{0}5_{1}\\[2pt]

2_{0}3_{1}4_{1}&0_{3}1_{1}6_{2}&2_{2}3_{2}4_{0}&0_{1}2_{3}6_{3}&0_{2}5_{3}6_{1}&1_{0}3_{0}6_{0}&4_{2}4_{3}5_{0}&1_{2}2_{1}3_{3}&1_{3}5_{1}5_{2}\\[2pt]

2_{0}3_{1}4_{1}&0_{1}1_{2}3_{2}&0_{2}3_{3}5_{1}&1_{3}2_{3}5_{0}&2_{1}4_{0}6_{0}&1_{1}4_{2}5_{3}&1_{0}4_{3}6_{3}&0_{3}2_{2}6_{1}&3_{0}5_{2}6_{2}\\[2pt]

\underline{2_{2}3_{0}4_{2}}&2_{3}4_{1}5_{2}&0_{2}6_{0}6_{2}&1_{2}6_{1}6_{3}&0_{3}5_{0}5_{3}&1_{0}2_{0}4_{0}&2_{1}3_{1}5_{1}&0_{1}1_{3}3_{3}&1_{1}3_{2}4_{3}\\[2pt]

2_{2}3_{0}4_{2}&0_{1}4_{0}4_{1}&4_{3}6_{1}6_{2}&3_{3}5_{0}6_{3}&0_{2}1_{0}1_{2}&2_{1}3_{2}5_{3}&1_{1}3_{1}6_{0}&0_{3}2_{0}2_{3}&1_{3}5_{1}5_{2}\\[2pt]

2_{2}3_{0}4_{2}&1_{3}2_{0}5_{3}&0_{1}4_{0}4_{1}&0_{3}6_{0}6_{3}&4_{3}6_{1}6_{2}&2_{1}3_{1}5_{1}&1_{1}1_{2}2_{3}&0_{2}5_{0}5_{2}&1_{0}3_{2}3_{3}\\[2pt]

\underline{3_{2}5_{1}6_{3}}&1_{1}2_{0}6_{1}&2_{1}4_{0}6_{0}&0_{3}1_{0}1_{3}&3_{1}3_{3}4_{2}&1_{2}4_{1}5_{3}&0_{1}2_{2}6_{2}&0_{2}5_{0}5_{2}&2_{3}3_{0}4_{3}\\[2pt]

3_{2}5_{1}6_{3}&3_{1}5_{3}6_{2}&1_{1}2_{0}6_{1}&2_{2}2_{3}6_{0}&0_{2}4_{1}4_{3}&0_{3}1_{0}1_{3}&3_{0}4_{0}5_{0}&1_{2}2_{1}3_{3}&0_{1}4_{2}5_{2}\\[2pt]

3_{2}5_{1}6_{3}&3_{1}5_{3}6_{2}&0_{2}4_{0}4_{2}&0_{1}1_{0}1_{1}&1_{2}1_{3}3_{0}&2_{2}2_{3}6_{0}&0_{3}2_{1}5_{2}&2_{0}3_{3}4_{3}&4_{1}5_{0}6_{1}\\[2pt]

\underline{1_{0}2_{3}5_{3}}&0_{1}1_{2}3_{2}&0_{3}3_{0}3_{3}&0_{2}2_{0}2_{2}&1_{3}2_{1}6_{2}&4_{1}5_{1}6_{0}&4_{2}4_{3}5_{0}&3_{1}4_{0}6_{1}&1_{1}5_{2}6_{3}\\[2pt]

1_{0}2_{3}5_{3}&0_{1}6_{0}6_{1}&1_{2}4_{3}5_{1}&0_{2}2_{0}2_{2}&3_{1}3_{3}4_{2}&1_{3}2_{1}6_{2}&3_{0}4_{0}5_{0}&0_{3}3_{2}4_{1}&1_{1}5_{2}6_{3}\\[2pt]

1_{0}2_{3}5_{3}&0_{1}6_{0}6_{1}&3_{2}5_{0}6_{2}&2_{1}2_{2}4_{3}&0_{3}3_{0}3_{3}&0_{2}3_{1}6_{3}&1_{1}4_{0}5_{1}&1_{3}4_{1}4_{2}&1_{2}2_{0}5_{2}\\[2pt]

\underline{1_{1}2_{1}5_{0}}&0_{2}6_{0}6_{2}&2_{2}4_{1}6_{3}&4_{0}5_{2}5_{3}&1_{0}5_{1}6_{1}&2_{0}3_{2}4_{2}&0_{1}1_{3}3_{3}&0_{3}1_{2}3_{1}&2_{3}3_{0}4_{3}\\[2pt]

1_{1}2_{1}5_{0}&0_{1}4_{3}5_{3}&0_{2}3_{3}5_{1}&1_{2}1_{3}3_{0}&2_{2}4_{1}6_{3}&1_{0}4_{2}6_{2}&3_{1}4_{0}6_{1}&0_{3}2_{0}2_{3}&3_{2}5_{2}6_{0}\\[2pt]

1_{1}2_{1}5_{0}&2_{0}3_{0}5_{1}&3_{1}4_{3}5_{2}&2_{3}4_{2}6_{1}&3_{3}5_{3}6_{0}&0_{2}1_{0}1_{2}&0_{1}2_{2}6_{2}&1_{3}4_{0}6_{3}&0_{3}3_{2}4_{1}\\[2pt]

\underline{0_{1}0_{2}0_{3}}&3_{1}4_{3}5_{2}&1_{3}3_{2}6_{1}&1_{0}2_{0}4_{0}&3_{3}5_{3}6_{0}&2_{3}5_{1}6_{2}&1_{2}2_{2}5_{0}&1_{1}3_{0}4_{1}&2_{1}4_{2}6_{3}\\[2pt]

0_{1}0_{2}0_{3}&2_{1}3_{0}6_{1}&1_{2}4_{3}5_{1}&1_{1}4_{2}5_{3}&3_{3}4_{1}6_{2}&2_{3}3_{1}3_{2}&2_{0}5_{0}6_{0}&1_{0}2_{2}5_{2}&1_{3}4_{0}6_{3}\\[2pt]

0_{1}0_{2}0_{3}&3_{2}5_{0}6_{2}&2_{1}2_{2}4_{3}&3_{0}5_{3}6_{3}&1_{0}5_{1}6_{1}&1_{3}4_{1}4_{2}&2_{3}3_{3}4_{0}&1_{1}3_{1}6_{0}&1_{2}2_{0}5_{2}

\end{array}$$

point $0_{1}$:
$$\begin{array}{lllllllll}

 \underline{3_{1}5_{3}6_{3}}&3_{3}4_{0}4_{2}&0_{2}6_{1}6_{2}&3_{0}4_{1}6_{0}&1_{0}1_{3}3_{2}&1_{2}2_{0}4_{3}&0_{3}2_{1}2_{3}&1_{1}2_{2}5_{2}&0_{0}5_{0}5_{1}\\[2pt]

 3_{1}5_{3}6_{3}&3_{0}3_{3}6_{2}&2_{1}4_{0}6_{1}&0_{2}2_{3}5_{0}&1_{0}5_{1}6_{0}&1_{2}2_{0}4_{3}&0_{3}1_{1}1_{3}&0_{0}4_{2}5_{2}&2_{2}3_{2}4_{1}\\[2pt]

 3_{1}5_{3}6_{3}&0_{0}1_{2}3_{2}&1_{0}4_{1}5_{0}&0_{2}2_{0}3_{3}&2_{2}2_{3}6_{1}&1_{1}2_{1}5_{1}&1_{3}3_{0}4_{2}&0_{3}6_{0}6_{2}&4_{0}4_{3}5_{2}\\[2pt]

 \underline{1_{3}2_{0}5_{2}}&3_{0}3_{3}6_{2}&2_{1}4_{0}6_{1}&0_{0}2_{3}6_{3}&0_{3}5_{1}5_{3}&3_{2}4_{3}6_{0}&1_{0}4_{1}5_{0}&2_{2}3_{1}4_{2}&0_{2}1_{1}1_{2}\\[2pt]

 1_{3}2_{0}5_{2}&0_{0}6_{0}6_{1}&3_{3}4_{0}4_{2}&1_{2}3_{0}5_{3}&2_{1}3_{1}4_{1}&0_{2}2_{3}5_{0}&3_{2}5_{1}6_{2}&0_{3}1_{0}2_{2}&1_{1}4_{3}6_{3}\\[2pt]

 1_{3}2_{0}5_{2}&0_{0}2_{3}6_{3}&1_{2}4_{2}6_{1}&3_{0}4_{0}5_{1}&0_{2}1_{0}5_{3}&1_{1}3_{1}5_{0}&2_{1}3_{3}4_{3}&0_{3}6_{0}6_{2}&2_{2}3_{2}4_{1}\\[2pt]

 \underline{1_{2}2_{2}5_{1}}&2_{1}5_{0}6_{0}&0_{2}4_{1}4_{2}&0_{0}1_{0}1_{1}&4_{0}5_{3}6_{2}&2_{3}3_{0}5_{2}&0_{3}3_{1}3_{3}&1_{3}4_{3}6_{1}&2_{0}3_{2}6_{3}\\[2pt]

 1_{2}2_{2}5_{1}&0_{3}4_{1}4_{3}&3_{2}5_{2}6_{1}&2_{0}2_{3}6_{2}&1_{1}3_{1}5_{0}&0_{0}1_{3}3_{3}&4_{2}5_{3}6_{0}&1_{0}2_{1}3_{0}&0_{2}4_{0}6_{3}\\[2pt]

 1_{2}2_{2}5_{1}&0_{3}5_{0}5_{2}&2_{1}3_{1}4_{1}&3_{2}4_{3}6_{0}&2_{0}2_{3}6_{2}&0_{0}1_{0}1_{1}&3_{3}5_{3}6_{1}&1_{3}3_{0}4_{2}&0_{2}4_{0}6_{3}\\[2pt]

 \underline{3_{0}5_{0}6_{1}}&0_{0}4_{3}5_{3}&0_{2}2_{1}2_{2}&1_{2}4_{1}6_{2}&1_{0}3_{1}4_{0}&1_{1}3_{2}3_{3}&1_{3}2_{3}5_{1}&5_{2}6_{0}6_{3}&0_{3}2_{0}4_{2}\\[2pt]

 3_{0}5_{0}6_{1}&0_{2}2_{1}2_{2}&0_{3}4_{1}4_{3}&1_{0}1_{2}6_{3}&2_{0}3_{1}6_{0}&4_{0}5_{3}6_{2}&1_{1}3_{2}3_{3}&1_{3}2_{3}5_{1}&0_{0}4_{2}5_{2}\\[2pt]

 3_{0}5_{0}6_{1}&2_{1}6_{2}6_{3}&4_{2}4_{3}5_{1}&0_{2}3_{1}3_{2}&1_{2}2_{3}6_{0}&0_{3}1_{1}1_{3}&1_{0}3_{3}5_{2}&0_{0}4_{0}4_{1}&2_{0}2_{2}5_{3}\\[2pt]

 \underline{1_{1}4_{0}6_{0}}&1_{3}5_{0}6_{2}&0_{2}4_{1}4_{2}&1_{2}3_{0}5_{3}&3_{2}5_{2}6_{1}&2_{3}3_{1}4_{3}&0_{0}2_{0}2_{1}&3_{3}5_{1}6_{3}&0_{3}1_{0}2_{2}\\[2pt]

 1_{1}4_{0}6_{0}&2_{2}4_{3}5_{0}&0_{2}6_{1}6_{2}&4_{1}5_{2}5_{3}&0_{0}2_{0}2_{1}&3_{3}5_{1}6_{3}&1_{0}2_{3}4_{2}&1_{2}1_{3}3_{1}&0_{3}3_{0}3_{2}\\[2pt]

 1_{1}4_{0}6_{0}&0_{0}1_{2}3_{2}&1_{3}5_{0}6_{2}&0_{3}6_{1}6_{3}&0_{2}1_{0}5_{3}&2_{3}3_{0}5_{2}&2_{1}3_{3}4_{3}&2_{0}4_{1}5_{1}&2_{2}3_{1}4_{2}\\[2pt]

 \underline{1_{0}4_{3}6_{2}}&1_{3}2_{2}4_{0}&1_{2}3_{3}5_{0}&1_{1}4_{1}6_{1}&0_{2}5_{1}5_{2}&2_{0}3_{2}6_{3}&0_{3}2_{1}2_{3}&0_{0}3_{0}3_{1}&4_{2}5_{3}6_{0}\\[2pt]

 1_{0}4_{3}6_{2}&2_{1}5_{0}6_{0}&4_{1}5_{2}5_{3}&2_{3}3_{2}4_{0}&0_{2}1_{1}1_{2}&3_{1}5_{1}6_{1}&2_{2}3_{0}6_{3}&0_{0}1_{3}3_{3}&0_{3}2_{0}4_{2}\\[2pt]

 1_{0}4_{3}6_{2}&4_{2}5_{0}6_{3}&0_{2}2_{0}3_{3}&1_{2}2_{3}6_{0}&1_{1}2_{2}5_{2}&1_{3}2_{1}5_{3}&3_{1}5_{1}6_{1}&0_{0}4_{0}4_{1}&0_{3}3_{0}3_{2}\\[2pt]

 \underline{2_{3}3_{3}4_{1}}&1_{2}2_{1}5_{2}&0_{2}1_{3}6_{0}&1_{0}3_{1}4_{0}&0_{3}6_{1}6_{3}&4_{2}4_{3}5_{1}&1_{1}2_{0}3_{0}&3_{2}5_{0}5_{3}&0_{0}2_{2}6_{2}\\[2pt]

 2_{3}3_{3}4_{1}&2_{1}6_{2}6_{3}&0_{0}4_{3}5_{3}&1_{3}2_{2}4_{0}&0_{3}5_{0}5_{2}&1_{2}4_{2}6_{1}&0_{2}3_{1}3_{2}&1_{0}5_{1}6_{0}&1_{1}2_{0}3_{0}\\[2pt]

 2_{3}3_{3}4_{1}&0_{0}6_{0}6_{1}&0_{3}1_{2}4_{0}&3_{1}5_{2}6_{2}&4_{2}5_{0}6_{3}&1_{0}1_{3}3_{2}&1_{1}2_{1}5_{1}&0_{2}3_{0}4_{3}&2_{0}2_{2}5_{3}\\[2pt]

  \end{array}$$
$$\begin{array}{lllllllll}

 \underline{2_{1}3_{2}4_{2}}&0_{3}1_{2}4_{0}&3_{1}5_{2}6_{2}&2_{2}3_{3}6_{0}&1_{1}2_{3}5_{3}&1_{0}2_{0}6_{1}&0_{2}3_{0}4_{3}&1_{3}4_{1}6_{3}&0_{0}5_{0}5_{1}\\[2pt]

 2_{1}3_{2}4_{2}&2_{0}4_{0}5_{0}&1_{0}1_{2}6_{3}&3_{0}4_{1}6_{0}&0_{2}5_{1}5_{2}&0_{3}3_{1}3_{3}&1_{1}2_{3}5_{3}&1_{3}4_{3}6_{1}&0_{0}2_{2}6_{2}\\[2pt]

 2_{1}3_{2}4_{2}&2_{0}4_{0}5_{0}&0_{2}1_{3}6_{0}&1_{2}4_{1}6_{2}&0_{3}5_{1}5_{3}&2_{2}2_{3}6_{1}&1_{0}3_{3}5_{2}&1_{1}4_{3}6_{3}&0_{0}3_{0}3_{1}\\[2pt]

 \underline{0_{0}0_{2}0_{3}}&1_{2}2_{1}5_{2}&3_{0}4_{0}5_{1}&2_{3}3_{1}4_{3}&2_{2}3_{3}6_{0}&3_{2}5_{0}5_{3}&1_{0}2_{0}6_{1}&1_{3}4_{1}6_{3}&1_{1}4_{2}6_{2}\\[2pt]

 0_{0}0_{2}0_{3}&1_{2}3_{3}5_{0}&1_{1}4_{1}6_{1}&2_{0}3_{1}6_{0}&3_{2}5_{1}6_{2}&1_{0}2_{3}4_{2}&1_{3}2_{1}5_{3}&2_{2}3_{0}6_{3}&4_{0}4_{3}5_{2}\\[2pt]

 0_{0}0_{2}0_{3}&2_{2}4_{3}5_{0}&3_{3}5_{3}6_{1}&2_{0}4_{1}5_{1}&1_{2}1_{3}3_{1}&2_{3}3_{2}4_{0}&5_{2}6_{0}6_{3}&1_{0}2_{1}3_{0}&1_{1}4_{2}6_{2}

\end{array}$$

Point $0_{2}$:
$$\begin{array}{lllllllll}

 \underline{1_{2}5_{3}6_{3}}&1_{3}2_{2}4_{3}&1_{0}2_{3}3_{1}&3_{3}4_{1}6_{0}&2_{0}2_{1}6_{2}&0_{0}3_{0}3_{2}&0_{3}4_{0}5_{0}&0_{1}5_{1}5_{2}&1_{1}4_{2}6_{1}\\[2pt]

 1_{2}5_{3}6_{3}&3_{2}4_{2}5_{2}&1_{3}3_{0}4_{1}&1_{1}2_{3}6_{0}&2_{0}4_{3}6_{1}&2_{1}3_{3}4_{0}&0_{0}0_{1}0_{3}&1_{0}2_{2}5_{0}&3_{1}5_{1}6_{2}\\[2pt]

 1_{2}5_{3}6_{3}&1_{3}2_{3}6_{2}&0_{0}4_{1}4_{3}&1_{0}3_{3}6_{1}&4_{2}5_{0}5_{1}&0_{3}2_{0}6_{0}&2_{2}3_{0}4_{0}&0_{1}3_{1}3_{2}&1_{1}2_{1}5_{2}\\[2pt]

 \underline{1_{3}3_{3}5_{2}}&0_{3}4_{2}4_{3}&2_{3}4_{0}6_{1}&2_{0}5_{1}6_{3}&4_{1}5_{0}5_{3}&1_{0}1_{1}3_{2}&1_{2}3_{0}3_{1}&0_{0}6_{0}6_{2}&0_{1}2_{1}2_{2}\\[2pt]

 1_{3}3_{3}5_{2}&0_{3}6_{2}6_{3}&2_{0}2_{3}4_{1}&2_{1}4_{3}5_{0}&5_{1}5_{3}6_{0}&2_{2}3_{0}4_{0}&0_{0}1_{0}1_{2}&1_{1}4_{2}6_{1}&0_{1}3_{1}3_{2}\\[2pt]

 1_{3}3_{3}5_{2}&2_{2}3_{2}6_{3}&2_{3}3_{0}5_{1}&3_{1}4_{3}6_{0}&0_{1}1_{0}5_{3}&0_{0}4_{0}4_{2}&1_{2}2_{0}5_{0}&0_{3}2_{1}6_{1}&1_{1}4_{1}6_{2}\\[2pt]

 \underline{0_{0}2_{1}2_{3}}&3_{3}4_{3}6_{3}&1_{2}2_{2}4_{2}&1_{3}2_{0}3_{1}&0_{1}1_{0}5_{3}&3_{0}5_{2}6_{0}&0_{3}4_{0}5_{0}&1_{1}4_{1}6_{2}&3_{2}5_{1}6_{1}\\[2pt]

 0_{0}2_{1}2_{3}&2_{2}3_{3}5_{3}&1_{3}4_{2}6_{3}&2_{0}4_{3}6_{1}&4_{0}4_{1}5_{2}&0_{3}1_{0}3_{0}&3_{2}5_{0}6_{0}&3_{1}5_{1}6_{2}&0_{1}1_{1}1_{2}\\[2pt]

 0_{0}2_{1}2_{3}&1_{2}2_{2}4_{2}&0_{3}5_{2}5_{3}&1_{3}2_{0}3_{1}&0_{1}3_{0}4_{3}&1_{1}5_{0}6_{3}&3_{3}4_{1}6_{0}&1_{0}4_{0}6_{2}&3_{2}5_{1}6_{1}\\[2pt]

 \underline{1_{0}4_{3}5_{1}}&2_{2}3_{3}5_{3}&0_{3}6_{2}6_{3}&0_{0}1_{1}1_{3}&0_{1}2_{3}5_{0}&1_{2}4_{0}6_{0}&2_{0}3_{0}4_{2}&2_{1}3_{2}4_{1}&3_{1}5_{2}6_{1}\\[2pt]

 1_{0}4_{3}5_{1}&2_{2}5_{2}6_{2}&0_{3}3_{2}3_{3}&0_{0}1_{1}1_{3}&0_{1}2_{3}5_{0}&2_{1}6_{0}6_{3}&3_{1}4_{0}5_{3}&2_{0}3_{0}4_{2}&1_{2}4_{1}6_{1}\\[2pt]

 1_{0}4_{3}5_{1}&2_{2}5_{2}6_{2}&0_{3}1_{2}1_{3}&2_{0}2_{3}4_{1}&0_{1}4_{0}6_{3}&1_{1}3_{0}3_{3}&0_{0}5_{3}6_{1}&3_{2}5_{0}6_{0}&2_{1}3_{1}4_{2}\\[2pt]

 \underline{2_{2}6_{0}6_{1}}&3_{3}4_{2}6_{2}&2_{3}3_{2}4_{3}&1_{3}3_{0}4_{1}&1_{1}5_{0}6_{3}&3_{1}4_{0}5_{3}&0_{0}0_{1}0_{3}&1_{0}2_{0}5_{2}&1_{2}2_{1}5_{1}\\[2pt]

 2_{2}6_{0}6_{1}&3_{3}4_{2}6_{2}&2_{3}3_{2}4_{3}&1_{0}1_{3}2_{1}&0_{1}4_{0}6_{3}&1_{1}2_{0}5_{3}&1_{2}3_{0}3_{1}&0_{0}5_{0}5_{2}&0_{3}4_{1}5_{1}\\[2pt]

 2_{2}6_{0}6_{1}&1_{2}3_{2}6_{2}&1_{3}4_{2}6_{3}&1_{0}2_{3}3_{1}&1_{1}4_{0}4_{3}&0_{1}2_{0}3_{3}&2_{1}3_{0}5_{3}&0_{0}5_{0}5_{2}&0_{3}4_{1}5_{1}\\[2pt]

 \underline{2_{0}3_{2}4_{0}}&1_{2}2_{3}3_{3}&0_{3}5_{2}5_{3}&1_{3}5_{0}6_{1}&0_{1}3_{0}4_{3}&1_{0}4_{1}6_{3}&0_{0}6_{0}6_{2}&1_{1}2_{2}5_{1}&2_{1}3_{1}4_{2}\\[2pt]

 2_{0}3_{2}4_{0}&4_{3}5_{3}6_{2}&2_{3}5_{2}6_{3}&0_{1}1_{3}6_{0}&1_{1}3_{0}3_{3}&4_{2}5_{0}5_{1}&0_{0}1_{0}1_{2}&0_{3}2_{1}6_{1}&2_{2}3_{1}4_{1}\\[2pt]

 2_{0}3_{2}4_{0}&2_{3}5_{2}6_{3}&0_{3}1_{2}1_{3}&0_{0}4_{1}4_{3}&3_{1}3_{3}5_{0}&2_{1}3_{0}5_{3}&1_{0}4_{2}6_{0}&1_{1}2_{2}5_{1}&0_{1}6_{1}6_{2}\\[2pt]

 \underline{3_{0}5_{0}6_{2}}&1_{3}2_{2}4_{3}&2_{3}4_{2}5_{3}&0_{0}3_{1}6_{3}&2_{1}3_{3}4_{0}&1_{0}1_{1}3_{2}&0_{3}2_{0}6_{0}&0_{1}5_{1}5_{2}&1_{2}4_{1}6_{1}\\[2pt]

 3_{0}5_{0}6_{2}&2_{3}4_{2}5_{3}&0_{3}3_{2}3_{3}&0_{1}1_{3}6_{0}&1_{1}4_{0}4_{3}&1_{0}4_{1}6_{3}&0_{0}2_{0}2_{2}&3_{1}5_{2}6_{1}&1_{2}2_{1}5_{1}\\[2pt]

 3_{0}5_{0}6_{2}&1_{3}3_{2}5_{3}&0_{3}4_{2}4_{3}&2_{3}4_{0}6_{1}&2_{1}6_{0}6_{3}&0_{0}3_{3}5_{1}&1_{0}2_{0}5_{2}&2_{2}3_{1}4_{1}&0_{1}1_{1}1_{2}\\[2pt]

 \end{array}$$
$$\begin{array}{lllllllll}

 \underline{0_{3}1_{1}3_{1}}&1_{2}4_{3}5_{2}&1_{3}2_{3}6_{2}&3_{0}6_{1}6_{3}&0_{1}2_{0}3_{3}&5_{1}5_{3}6_{0}&1_{0}2_{2}5_{0}&0_{0}4_{0}4_{2}&2_{1}3_{2}4_{1}\\[2pt]

 0_{3}1_{1}3_{1}&4_{3}5_{3}6_{2}&1_{2}2_{3}3_{3}&1_{3}5_{0}6_{1}&2_{0}5_{1}6_{3}&4_{0}4_{1}5_{2}&1_{0}4_{2}6_{0}&0_{0}3_{0}3_{2}&0_{1}2_{1}2_{2}\\[2pt]

 0_{3}1_{1}3_{1}&3_{3}4_{3}6_{3}&3_{2}4_{2}5_{2}&1_{0}1_{3}2_{1}&2_{3}3_{0}5_{1}&4_{1}5_{0}5_{3}&0_{0}2_{0}2_{2}&1_{2}4_{0}6_{0}&0_{1}6_{1}6_{2}\\[2pt]

 \underline{0_{1}4_{1}4_{2}}&2_{2}3_{2}6_{3}&1_{2}4_{3}5_{2}&1_{3}4_{0}5_{1}&1_{1}2_{3}6_{0}&3_{1}3_{3}5_{0}&0_{0}5_{3}6_{1}&2_{0}2_{1}6_{2}&0_{3}1_{0}3_{0}\\[2pt]

 0_{1}4_{1}4_{2}&1_{2}3_{2}6_{2}&0_{3}2_{2}2_{3}&1_{3}4_{0}5_{1}&2_{1}4_{3}5_{0}&0_{0}3_{1}6_{3}&1_{0}3_{3}6_{1}&1_{1}2_{0}5_{3}&3_{0}5_{2}6_{0}\\[2pt]

 0_{1}4_{1}4_{2}&1_{3}3_{2}5_{3}&0_{3}2_{2}2_{3}&3_{1}4_{3}6_{0}&3_{0}6_{1}6_{3}&0_{0}3_{3}5_{1}&1_{0}4_{0}6_{2}&1_{2}2_{0}5_{0}&1_{1}2_{1}5_{2}

\end{array}$$

Point $0_{3}$:
$$\begin{array}{lllllllll}

 \underline{3_{2}5_{3}6_{2}}&2_{2}3_{3}5_{2}&1_{2}3_{0}6_{1}&0_{0}4_{2}5_{1}&2_{0}2_{1}6_{3}&1_{3}4_{0}6_{0}&1_{0}2_{3}5_{0}&0_{2}1_{1}3_{1}&0_{1}4_{1}4_{3}\\[2pt]

 3_{2}5_{3}6_{2}&1_{3}3_{3}4_{2}&4_{0}5_{2}6_{1}&0_{0}1_{2}3_{1}&1_{1}2_{0}2_{2}&1_{0}4_{3}6_{0}&3_{0}5_{0}6_{3}&0_{1}2_{1}2_{3}&0_{2}4_{1}5_{1}\\[2pt]

 3_{2}5_{3}6_{2}&0_{2}1_{2}1_{3}&4_{0}5_{2}6_{1}&0_{1}2_{0}4_{2}&2_{1}2_{2}3_{0}&1_{0}1_{1}3_{3}&4_{3}5_{0}5_{1}&0_{0}6_{0}6_{3}&2_{3}3_{1}4_{1}\\[2pt]

 \underline{1_{2}2_{3}5_{2}}&1_{3}2_{2}3_{2}&4_{0}5_{1}6_{2}&4_{1}4_{2}6_{0}&2_{0}2_{1}6_{3}&0_{2}1_{0}3_{0}&0_{0}5_{0}5_{3}&1_{1}4_{3}6_{1}&0_{1}3_{1}3_{3}\\[2pt]

 1_{2}2_{3}5_{2}&0_{2}4_{2}4_{3}&1_{0}2_{1}6_{2}&0_{0}3_{2}4_{1}&2_{2}5_{1}6_{0}&2_{0}3_{3}4_{0}&3_{0}5_{0}6_{3}&3_{1}5_{3}6_{1}&0_{1}1_{1}1_{3}\\[2pt]

 1_{2}2_{3}5_{2}&1_{3}3_{3}4_{2}&0_{1}6_{0}6_{2}&2_{0}3_{2}6_{1}&2_{2}4_{1}5_{0}&0_{2}1_{0}3_{0}&0_{0}4_{0}4_{3}&3_{1}5_{1}6_{3}&1_{1}2_{1}5_{3}\\[2pt]

 \underline{2_{2}3_{1}4_{0}}&4_{2}5_{2}6_{3}&0_{2}1_{2}1_{3}&0_{0}1_{1}6_{2}&0_{1}3_{0}3_{2}&2_{3}6_{0}6_{1}&4_{3}5_{0}5_{1}&1_{0}2_{0}5_{3}&2_{1}3_{3}4_{1}\\[2pt]

 2_{2}3_{1}4_{0}&2_{3}4_{2}6_{2}&1_{3}4_{3}6_{3}&0_{1}5_{0}5_{2}&0_{0}3_{2}4_{1}&1_{2}2_{0}5_{1}&1_{0}1_{1}3_{3}&3_{0}5_{3}6_{0}&0_{2}2_{1}6_{1}\\[2pt]

 2_{2}3_{1}4_{0}&3_{3}5_{3}6_{3}&0_{2}4_{2}4_{3}&5_{0}6_{1}6_{2}&2_{0}4_{1}5_{2}&0_{1}3_{0}3_{2}&1_{2}2_{1}6_{0}&0_{0}1_{0}1_{3}&1_{1}2_{3}5_{1}\\[2pt]

 \underline{1_{0}4_{2}6_{1}}&1_{2}4_{3}5_{3}&1_{3}5_{2}6_{2}&2_{1}3_{2}5_{0}&1_{1}2_{0}2_{2}&2_{3}3_{0}4_{0}&0_{0}6_{0}6_{3}&0_{1}3_{1}3_{3}&0_{2}4_{1}5_{1}\\[2pt]

 1_{0}4_{2}6_{1}&1_{2}2_{2}6_{3}&0_{2}5_{2}5_{3}&2_{0}3_{1}6_{2}&2_{1}3_{2}5_{0}&1_{3}4_{0}6_{0}&0_{0}3_{0}3_{3}&1_{1}2_{3}5_{1}&0_{1}4_{1}4_{3}\\[2pt]

 1_{0}4_{2}6_{1}&1_{3}4_{3}6_{3}&0_{2}2_{2}2_{3}&3_{0}4_{1}6_{2}&0_{0}2_{1}5_{2}&3_{1}3_{2}6_{0}&1_{1}1_{2}5_{0}&2_{0}3_{3}4_{0}&0_{1}5_{1}5_{3}\\[2pt]

 \underline{0_{0}0_{1}0_{2}}&2_{3}3_{2}6_{3}&2_{2}4_{3}6_{2}&1_{0}3_{1}5_{2}&1_{2}2_{1}6_{0}&1_{1}3_{0}4_{2}&4_{0}4_{1}5_{3}&1_{3}2_{0}5_{0}&3_{3}5_{1}6_{1}\\[2pt]

 0_{0}0_{1}0_{2}&1_{2}2_{2}6_{3}&2_{3}3_{3}4_{3}&1_{0}2_{1}6_{2}&3_{0}5_{1}5_{2}&1_{1}3_{2}4_{0}&4_{1}4_{2}6_{0}&1_{3}2_{0}5_{0}&3_{1}5_{3}6_{1}\\[2pt]

 0_{0}0_{1}0_{2}&2_{2}3_{3}5_{2}&1_{3}2_{3}5_{3}&3_{0}4_{1}6_{2}&2_{0}3_{2}6_{1}&1_{1}1_{2}5_{0}&2_{1}4_{0}4_{2}&1_{0}4_{3}6_{0}&3_{1}5_{1}6_{3}\\[2pt]

 \underline{3_{3}5_{0}6_{0}}&3_{2}4_{3}5_{2}&2_{3}4_{2}6_{2}&1_{2}2_{0}5_{1}&2_{1}2_{2}3_{0}&4_{0}4_{1}5_{3}&0_{0}1_{0}1_{3}&0_{2}1_{1}3_{1}&0_{1}6_{1}6_{3}\\[2pt]

 3_{3}5_{0}6_{0}&2_{3}3_{2}6_{3}&2_{2}4_{2}5_{3}&2_{0}3_{1}6_{2}&3_{0}5_{1}5_{2}&1_{0}1_{2}4_{1}&0_{0}4_{0}4_{3}&0_{2}2_{1}6_{1}&0_{1}1_{1}1_{3}\\[2pt]

 3_{3}5_{0}6_{0}&1_{3}2_{2}3_{2}&0_{2}5_{2}5_{3}&4_{0}5_{1}6_{2}&1_{0}1_{2}4_{1}&1_{1}3_{0}4_{2}&0_{0}2_{0}2_{3}&2_{1}3_{1}4_{3}&0_{1}6_{1}6_{3}\\[2pt]

 \underline{2_{0}3_{0}4_{3}}&1_{2}3_{3}6_{2}&0_{2}2_{2}2_{3}&0_{1}5_{0}5_{2}&3_{1}3_{2}6_{0}&0_{0}4_{2}5_{1}&1_{0}4_{0}6_{3}&1_{3}4_{1}6_{1}&1_{1}2_{1}5_{3}\\[2pt]

 2_{0}3_{0}4_{3}&1_{2}3_{2}4_{2}&3_{3}5_{3}6_{3}&0_{0}1_{1}6_{2}&1_{0}3_{1}5_{2}&2_{2}5_{1}6_{0}&0_{2}4_{0}5_{0}&0_{1}2_{1}2_{3}&1_{3}4_{1}6_{1}\\[2pt]

 2_{0}3_{0}4_{3}&1_{3}2_{3}5_{3}&0_{2}6_{2}6_{3}&1_{1}5_{2}6_{0}&1_{0}3_{2}5_{1}&0_{1}1_{2}4_{0}&3_{1}4_{2}5_{0}&0_{0}2_{2}6_{1}&2_{1}3_{3}4_{1}\\[2pt]

 \underline{1_{1}4_{1}6_{3}}&3_{2}4_{3}5_{2}&1_{2}3_{3}6_{2}&2_{1}4_{0}4_{2}&0_{0}2_{2}6_{1}&1_{3}3_{0}3_{1}&0_{2}2_{0}6_{0}&1_{0}2_{3}5_{0}&0_{1}5_{1}5_{3}\\[2pt]

 1_{1}4_{1}6_{3}&2_{2}4_{2}5_{3}&2_{3}3_{3}4_{3}&5_{0}6_{1}6_{2}&0_{0}2_{1}5_{2}&1_{0}3_{2}5_{1}&0_{1}1_{2}4_{0}&1_{3}3_{0}3_{1}&0_{2}2_{0}6_{0}\\[2pt]

 1_{1}4_{1}6_{3}&1_{2}3_{2}4_{2}&1_{3}5_{2}6_{2}&0_{1}1_{0}2_{2}&0_{0}2_{0}2_{3}&3_{0}5_{3}6_{0}&0_{2}4_{0}5_{0}&3_{3}5_{1}6_{1}&2_{1}3_{1}4_{3}\\[2pt]

 \underline{1_{3}2_{1}5_{1}}&1_{2}4_{3}5_{3}&0_{2}6_{2}6_{3}&2_{0}4_{1}5_{2}&1_{1}3_{2}4_{0}&3_{1}4_{2}5_{0}&0_{1}1_{0}2_{2}&2_{3}6_{0}6_{1}&0_{0}3_{0}3_{3}\\[2pt]

 1_{3}2_{1}5_{1}&2_{2}4_{3}6_{2}&0_{2}3_{2}3_{3}&1_{1}5_{2}6_{0}&1_{2}3_{0}6_{1}&0_{1}2_{0}4_{2}&0_{0}5_{0}5_{3}&1_{0}4_{0}6_{3}&2_{3}3_{1}4_{1}\\[2pt]

 1_{3}2_{1}5_{1}&4_{2}5_{2}6_{3}&0_{2}3_{2}3_{3}&0_{1}6_{0}6_{2}&0_{0}1_{2}3_{1}&2_{2}4_{1}5_{0}&1_{0}2_{0}5_{3}&2_{3}3_{0}4_{0}&1_{1}4_{3}6_{1}

\end{array}$$

\end{document}